\documentclass[11pt,oneside,reqno,english]{amsart}
\usepackage{amsmath, amsfonts, latexsym, array, amssymb, amscd}
\usepackage{graphicx}
\usepackage[pagewise]{lineno}
\usepackage{color}
\usepackage{enumitem}
\usepackage{epstopdf}
\usepackage[top=1in, bottom=1in, left = 1.25in, right=1.25in, includefoot, headheight=13.6pt]{geometry}
\usepackage{hyperref}
\usepackage{mathtools}
\usepackage{mathrsfs}
\usepackage{parskip}
\usepackage{accents}
\usepackage[svgnames]{xcolor}

\numberwithin{equation}{section}
\numberwithin{figure}{section}

\newtheorem{theorem}{Theorem}[section]
\newtheorem*{theorem*}{Theorem}
\newtheorem{lemma}[theorem]{Lemma}

\newtheorem{proposition}[theorem]{Proposition}
\newtheorem*{proposition*}{Proposition}

\theoremstyle{definition}
\newtheorem{definition}[theorem]{Definition}

\newtheorem{remark}[theorem]{Remark}

\usepackage{indentfirst}
\usepackage{ulem}

\address{Department of Mathematics, The University of Chicago, Chicago, IL 60637, USA} 
\email{kihopark@math.uchicago.edu}

\address{Department of Mathematics,
The Ohio State University, Columbus, OH 43210, USA} 
\email{wang.7828@osu.edu}

\newcommand{\R}{\mathbb{R}}
\newcommand{\D}{\mathcal{D}}
\newcommand{\G}{\mathcal{G}}
\newcommand{\E}{\mathcal{E}}
\newcommand{\J}{\mathcal{J}}
\newcommand{\Z}{\mathbb{Z}}
\newcommand{\N}{\mathbb{N}}
\renewcommand{\P}{\mathcal{P}}

\def\wt{\widetilde}
\def\wh{\widehat}

\def\a{\alpha}
\def\b{\beta}
\def\d{\delta}
\def\g{\gamma}
\def\ep{\varepsilon}
\def\F{\mathcal{F}}
\def\W{\mathcal{W}}
\def\L{\mathcal{L}}

\def\C{\mathcal{C}}
\def\s{\sigma}

\def\k{\kappa}
\def\w{\mathsf{w}}
\def\v{\mathsf{v}}
\def\u{\mathsf{u}}
\def\RR{\mathsf{R}}

\def\lt{\lambda_{T}}

\def\M{\mathcal{M}}

\def\Sing{\mathrm{Sing}}
\def\Reg{\mathrm{Reg}}
\def\vp{\varphi}
\def\vpgeo{\varphi^{geo}}

\title{Multifractal analysis of geodesic flows on surfaces without focal points}
\author{Kiho Park, Tianyu Wang \vspace{-2em}}
\date{\today}

\begin{document}

\maketitle
\begin{abstract}
We study multifractal spectra of the geodesic flows on rank 1 surfaces without focal points. We compute the entropy of the level sets for the Lyapunov exponents and estimate its Hausdorff dimension from below. In doing so, we employ and generalize results of Burns and Gelfert.
\end{abstract}

\section{introduction}
In this paper we study the multifractal information of the Lyapunov level sets with respect to the geodesic flow on surfaces with no focal points. In particular, we focus on estimating the topological entropy and the Hausdorff dimension of such level sets. Historically, similar types of problems have been studied in a greater generality for uniformly hyperbolic systems; see \cite{BD04} for flows, \cite{BS01,PW01} for some discrete time examples, and \cite{Pes97} for a systematic introduction. 

While there are some known results for one-dimensional \cite{JJOP10} and conformal systems \cite{Cli11}, in general much less is known regarding the multifractal analysis for non-uniformly hyperbolic systems. Higher dimensional generalization is more difficult because popular methods for estimating the Hausdorff dimension by slicing it with the stable and unstable leaves \cite{mccluskey1983hausdorff} and adding up the respective dimensions are no longer valid. Failure of such methods is due to the fact that the projection map along the holomony is often not bi-Lipschitz for non-uniformly hyperbolic systems, even when the stable and unstable leaves are absolutely continuous. 

In a different setting, Burns and Gelfert \cite{burns2014lyapunov} studied the Lyapunov level sets with respect to the geodesic flow over rank 1 non-positively curved surfaces. For such surfaces, there exists a closed invariant subset of the unit tangent bundle, called the \emph{singular set}, on which the geodesic flow experiences no hyperbolicity. It is the presence of such a singular set that makes the geodesic flow non-uniformly hyperbolic.

In this paper, we extend the results of \cite{burns2014lyapunov} to rank 1 surfaces without focal points by employing similar techniques. Manifolds without focal points are natural generalizations of non-positively curved manifolds, and two classes of manifolds share many geometric features such as the presence of the singular set. On the other hand, there are certain properties that only hold for non-positively curved manifolds; see Remark \ref{rem: difference in setting} for instance.

Let $S$ be a closed surface equipped with a Riemannian metric such that there are no focal points (see Definition \ref{defn: no focal}), and $G=\{g_t\}_{t\in \mathbb{R}}$ be the geodesic flow on its unit tangent bundle $T^1S$. We will assume throughout the paper that the singular set is non-empty because the geodesic flow is uniformly hyperbolic in the absence of the singular set, and the multifractal analysis of the Lyapunov level sets is then well-understood.

The \textit{geometric potential} is defined by
$$\vpgeo(v):=-\lim\limits_{t \to 0} \frac{1}{t}\log\Big\|dg_t|_{E^u_v}\Big\|$$
and its scalar multiples $t\vpgeo$ are intimately related to the \textit{Lyapunov exponent} because the Lyapunov exponent $\chi(v)$ can be defined as the Birkhoff average of $-\vpgeo$ along the orbit of $v$. We say $v\in T^1S$ is \emph{Lyapunov regular} if its Lyapunov exponent $\chi(v)$ exists; see Definition \ref{defn: Lyap exp}. 
The main object of study in this paper is the \emph{Lyapunov level set} defined as
$$\mathcal{L}(\b):=\{v \in T^1S \colon v \text{ is Lyapuonv regular and } \chi(v) = \b\}.$$

We denote by $h(\L(\b))$ the topological entropy of $\L(\b)$, also known as the \emph{entropy spectrum}. 
Here we adopt Bowen's definition \cite{bowen1973topological} of the entropy for non-compact sets as the Lyapunov level sets are non-compact in general. We also denote by $\dim_H\mathcal{L}(\b)$ the Hausdorff dimension of $\L(\b)$.
The pressure function and its Legendre transform
\begin{equation}\label{eq: P and E}
\mathcal{P}(t):=P(t\varphi^{geo})
~\text{ and }~
\mathcal{E}(\alpha) := \inf\limits_{t\in \R}\Big( \mathcal{P}(t)-t\alpha\Big),
\end{equation}
play important roles in computing $h(\L(\b))$ and $\dim_H\mathcal{L}(\b)$. 
Setting
$$\alpha_1:=\lim_{t\rightarrow -\infty}D^{+}\mathcal{P}(t),$$
whose limit is well-defined due to the monotonicity and convexity of $\mathcal{P}$, our main result is as follows:
\begin{theorem}\label{thm: main}  
Let $S$ be a rank 1 Riemannian surface with no focal points. Then,
\begin{enumerate}
\item
The Lyapunov level set $\L(-\a)$ is non-empty if and only if $\alpha \in [\a_1,0]$. 
\item For $\alpha \in (a_1,0)$, we have
$$h(\mathcal{L}(-\alpha)) = \mathcal{E}(\alpha)$$
and
$$\dim_H\mathcal{L}(-\alpha) \geq 1+ 2 \cdot \frac{\mathcal{E}(\a)}{-\a}.$$
\end{enumerate}
\end{theorem}

Given results of \cite{burns2014lyapunov} and the similarity of the setting there, in some sense Theorem \ref{thm: main} is an expected result. However, we believe it is still worthwhile to provide the proof as relevant modifications need to be made. Such modifications are based on recently introduced tools (see $\mathsection$ \ref{sec: proof}) suitable for studying manifolds without focal points.

We briefly comment on our approach to this result, which is based on \cite{burns2014lyapunov}.
Since the singular set is non-empty, the geometric potential $\varphi^{geo}$ has more than one equilibrium state. Consequently, the pressure function $\P(t)$ is not differentiable at $t=1$, exhibiting a \emph{phase transition}. 
Based on this observation, we study the Lyapunov level sets $\mathcal{L}(\beta)$ in two separate cases depending on the domain of $\b$. 

Setting $\displaystyle \alpha_2:=D^{-}\mathcal{P}(1)$,
the first case concerns with the domain $\beta\in (-\alpha_2,-\alpha_1)$ corresponding to the time before the phase transition. 
In this case, we compute the entropy $h(\L(\b))$
using the fact that $\mathcal{P}$ is $C^1$ when $t<1$, which is based on the uniqueness of the equilibrium state for $t\varphi^{geo}$ obtained in \cite{chen2020unique}; see $\mathsection$ \ref{sec: prelim} for further discussions.

The other case $\beta\in [0,-\alpha_2]$ corresponds to the time past the phase transition, where we estimate the entropy spectrum from both below and above. For the lower bound, we follow the construction in \cite{burns2014lyapunov},
and build an increasingly nested sequence of basic sets (see Definition \ref{defn: basic}) in the complement of the singular set on which the geodesic flow is non-uniformly hyperbolic; see Proposition \ref{prop: Lambda sequence}. As the entropy spectrum is well-understood on the basic sets, and the increasingly nested basic sets are constructed so that they eventually intersect $\L(\b)$ non-trivially, we use such information to establish an effective lower bound for the entropy of $\mathcal{L}(\beta)$. In our case, the construction of such a sequence of basic sets relies on the hyperbolic index function $\lambda_T$ introduced in \cite{chen2020unique}; see $\mathsection$ \ref{sec: proof}. 

The paper is structured as follows. In $\mathsection$ \ref{sec: prelim}, we survey preliminaries on geometry, thermodynamic formalism, and multifractal anlaysis. In $\mathsection$ \ref{sec: 3}, we prove Theorem \ref{thm: main} by assuming the existence of the increasingly nested basic sets (Proposition \ref{prop: Lambda sequence}). In $\mathsection$ \ref{sec: proof}, we establish Proposition \ref{prop: Lambda sequence}. In its proof, we make use of Proposition \ref{prop: key}, a key proposition that allows us to equip local maximality to given hyperbolic sets, whose proof is deferred to $\mathsection$ \ref{sec: 5}. \\

\noindent\textbf{Acknowledgment} The authors would like to thank American Institute of Mathematics where this project was initiated. The authors would also like to thank Amie Wilkinson, Dan Thompson, and Keith Burns for useful discussions.

\section{Preliminaries}\label{sec: prelim}
\subsection{Geometry} 
In this subsection, we introduce and survey geometric features of manifolds with no focal points.
Let $M$ be a closed Riemannian manifold. For $v\in T^1M$, we denote  by $\g_v$ the unique geodesic with the initial velocity $v \in T^1M$; in particular, the geodesic flow $G=\{g_t\}_{t\in \R}$ is defined as $g_t v := \dot{\gamma}_v(t)$. Often we will identify the orbit segment connecting $v$ to $g_t v$ with $(v,t)\in T^1M \times [0,\infty)$. We equip $T^1M$ with following metric
\begin{equation}\label{eq: Knieper}
d(v,w):=\max\{d(\g_v(t),\g_w(t)) \colon t \in [0,1]\}.
\end{equation}
This metric was used by Knieper \cite{knieper1998uniqueness}, and it is locally equivalent to the more commonly used Sasaki metric.

A \textit{Jacobi field} $J(t)$ along a geodesic $\gamma$ is a vector field along $\gamma$ satisfying the following \textit{Jacobi equation}:
$$
J''(t)+R(J(t),\dot{\gamma}(t))\dot{\gamma}(t)=0
$$
where $'$ denotes the covariant derivative along $\gamma$ and $R$ is the Riemmanian curvature tensor. 
We denote the space of all Jacobi fields along $\gamma$ by $\J(\gamma)$. 

It is clear from the Jacobi equation that the behaviors of the Jacobi fields are governed by the Riemannian curvature of the manifold. For negatively curved manifolds, the function $t \mapsto \|J(t)\|$ is strictly convex for any Jacobi field $J$. Non-positively curved manifolds are natural generalizations of negatively curved manifolds, and $t \mapsto \|J(t)\|$ is convex. 
In regards to these behaviors of the Jacobi fields, it is then clear from the following definition that manifolds with no focal points are natural generalizations of non-positively curved manifolds:

\begin{definition}\label{defn: no focal}
A manifold $M$ has \textit{no focal points} if for any initially vanishing Jacobi field $J(t)$, its length $\|J(t)\|$ is strictly increasing.
\end{definition}

We say a Jacobi field $J$ is \textit{orthogonal} if both $J(t_0)$ and $J'(t_0)$ are perpendicular to $\dot{\gamma}(t_0)$ for some $t_0 \in \R$. Note that any such Jacobi field $J$ with this property has an evidently stronger property that $J$ and $J'$ are perpendicular to $\dot{\gamma}$ for \textit{all} $t \in \R$; see \cite{carmo1992riemannian}. 
We say a Jacobi $J$ is \textit{parallel} if $J' \equiv 0$. 

\begin{definition} 
For $v \in T^1M$, the \textit{rank} of $v$ is the dimension of the space of parallel Jacobi fields over $\gamma_v$. We say the manifold is \textit{rank 1} if it has at least one rank 1 vector.
The \textit{singular set} is the set of vectors with rank bigger than or equal 2:
$$\Sing:=\big\{v\in T^1M \colon \text{rank}(v) \geq 2\big\}.$$
The \emph{regular set}, $\Reg:= T^1M \setminus \Sing$, is defined as the complement of the singular set.
\end{definition}

We also denote the space of orthogonal Jacobi fields along $\gamma$ by $\J^\perp(\gamma)$. The set of \textit{stable orthogonal Jacobi fields} $\J^s(\gamma)$ is a subspace of $\J^\perp(\g)$ defined by
$$\J^s(\gamma) =\{J \in \J^\perp(\gamma) \colon \|J(t)\| \text{ is bounded for }t \geq 0\}.$$
Similarly, the set of \textit{unstable orthogonal Jacobi fields} $\J^u(\gamma)$ consists of orthogonal Jacobi fields whose norm is bounded for all $t \leq 0$.
By pulling back these linear subspaces $\J^{s/u}$ via the identification $T_vT^1M \simeq \J(\gamma_v)$, we define subbundles $E^{s/u}$ in $TT^1M$ by
$$E^{s/u}(v):=\{\xi \in T_vT^1M \colon J_\xi \in \J^{s/u}(v)\},$$
where $J_\xi$ is the Jacobi field along $\g_v$ whose initial conditions $J_\xi(0)$ and $J'_\xi(0)$ are specified by $\xi$.
We also define a 1-dimensional subbundle $E^c \subset TT^1M$ given by the geodesic flow direction.

The following proposition summarizes known properties on these subbundles when $M$ has no focal points; these properties are collected from \cite{pesin1977geodesic, eschenburg1977horospheres, hurley1986ergodicity}.
\begin{proposition}\label{prop: properties}
 Let $M$ be a $n$-dimensional closed Riemannian manifold without focal points.
\begin{enumerate}
\item $\dim(E^{s/u}) = n-1$ and $\dim(E^c) = 1$.
\item The subbundles $E^s,E^u,E^{cs}:=E^c\oplus E^s$ and $E^{cu}:=E^c\oplus E^u$ are $G$-invariant.
\item The subbundles $E^\s,\s \in \{s,u,cs,cu\}$, are integrable to $G$-invariant foliations $W^\s$.
\item $E^s(v)$ and $E^u(v)$ have a non-trivial intersection if and only if $v \in \Sing$.
\item For any $J \in \J^s(\gamma)$ (respectively, $\J^u(\gamma$)), $\|J(t)\|$ is non-increasing (respectively, non-decreasing) for all $t \in \R$.
\item The geodesic flow $G  =\{g_t\}_{t\in \R}$ on $T^1M$ is topologically transitive if $M$ is rank 1.
\end{enumerate}
\end{proposition}

The foliations $W^\sigma$, $\sigma \in \{s,u,cs,cu\}$, defined in the above proposition are called \emph{stable}, \emph{unstable}, \emph{center-stable}, and \emph{center-unstable}, respectively.
We endow each foliation with the \emph{intrinsic metric} as follows. On $W^s(v)$ we define
$$d^s(u,w) :=\inf \{\ell(\pi \gamma) \mid \g \colon [0,1] \to W^s(v),~ \g(0)=u,~\g(1) = w\}$$
where $\pi\colon T^1M \to M$ is the canonical projection, $\ell$ denotes the length of the curve on $M$, and the infimum is taken over all $C^1$ curves $\g$ connecting $u$ and $v$ in $W^s(v)$. 
Locally on $W^{cs}(v)$, we define 
$$d^{cs}(u,w):=|t|+d^s(g_tu,w)$$
where $t\in \R$ is the unique real number such that $g_tu \in W^s(w)$.  It extends to the metric on the entire center-stable leaf $W^{cs}(v)$ in an obvious way.
Moreover, it follows from the above proposition that the map $t \mapsto d^{\sigma}(g_tu,g_tv)$ is non-increasing for $\s \in \{s,cs\}$. 

Likewise, the metrics $d^u$ and $d^{cu}$ on $W^u$ and $W^{cu}$, respectively, are defined analogously and the map $t \mapsto d^{\sigma}(g_tu,g_tv)$ is non-decreasing for $\s \in \{u,cu\}$. All intrinsic metrics $d^{\s}$, $\s\in \{s,u,cs,cu\}$, are locally equivalent to the metric $d$ from \eqref{eq: Knieper}.

In this paper, we will primarily focus on rank 1 surfaces $S$ without focal points. The rank 1 condition on $S$ is equivalent to its genus being at least 2,
and the singular set admits an alternative description  given by
\begin{equation}\label{eq: char Sing by K}
\Sing = \big\{v \in T^1S \colon K(\pi g_tv) = 0 \text{~for all~} t \in\R\big\},
\end{equation}
where $K$ is the Gaussian curvature; see \cite{hopf1948closed} and \cite[Corollary 3.3, 3.6]{eberlein1973geodesic}. 
Moreover, the Jacobi equation simplifies to
\begin{equation}\label{eq: jacobi surface}
J''(t)+K(\gamma(t))J(t)=0.
\end{equation}
Given any orthogonal Jacobi field $J$ along $\gamma$, we may identify it as a unit vector field along $\gamma$ scaled by a continuous function. By an abuse of notation, we denote such a continuous function also by $J$, and we may treat \eqref{eq: jacobi surface} as a scalar ODE.

For any $v\in T^1S$, the unstable leaf $W^u(v)$ projects to the horosphere $H^u(v) \subset S$, and the one-dimensional symmetric operator on $T_{\pi v}H^u(v)$ defines the \emph{geodesic curvature} $k^u(v)$ of the horosphere $H^u(v)$; see \cite[Section 3]{chen2020unique} for details. Likewise, the geodesic curvature $k^s(v)$ of the stable horosphere $H^s(v)$ is defined analogously, and satisfies $k^s(-v) = k^u(v)$. 

Using the geodesic curvatures $k^{s/u}$, we define a non-negative function $\lambda$ on $T^1S$ given by
$$\lambda(v) := \min\{k^s(v),k^u(v)\}.$$
The horospheres are $C^2$ for manifolds without focal points, so $k^{s/u}$ and $\lambda$ are continuous. Such $\lambda$ first appeared in \cite{burns2018unique} in the setting of non-positively curved manifolds as a mean of measuring hyperbolicity on the unit tangent bundle. It serves the same purpose in our setting, and satisfies the following properties:

\begin{proposition}\label{prop: lambda properties}
\begin{enumerate}
\item $\lambda|_{\Sing} \equiv 0$.
\item If $\lambda(v)=0$, then $K(\pi v)=0$.
\item If $\lambda(g_tv) = 0$ for all $t \in\R$, then $v \in \Sing$.
\end{enumerate}
\end{proposition}
\begin{proof}
The first two statements relies on \cite[Lemma 2.9]{burns2018unique} (see also \cite[Lemma 3.11]{chen2020unique}) which states that the unstable Jacobi field $J^u \in \J^u(\g_v)$ with $J^u(0)=1$ satisfies
\begin{equation}\label{eq: J and k}
(J^u)'(t) = k^u(g_tv) J^u(t) \text{ for all }t \in \R.
\end{equation}
Likewise, $(J^s)'(t) = -k^s(g_tv) J^s(t)$ for $J^s \in \J^s(\g_v)$ with $J^s(0)=1$.
 
For the first statement, if $v \in \Sing$, then $E^s(v)$ and $E^u(v)$ coincide, and the norm of the Jacobi field $J_\xi$ corresponding to $\xi \in E^s(v) = E^u(v)$ is constant from Proposition \ref{prop: properties}. Then it follows from \eqref{eq: J and k} and the definition of $\lambda$ that $\lambda(v)=0$.

For the second statement, let $v \in T^1S$ with $\lambda(v) = 0$, and without loss of generality, suppose that $k^s(v) = 0$. Then the stable Jacobi field $J^s\in \J^s(\g_v)$ with $J^s(0)=1$ 
satisfies $(J^s)'(0) = 0$ from \eqref{eq: J and k}. Since $(J^s)'$ is a non-positive function from Proposition \ref{prop: properties}, it follows that $(J^s)''(0)=0$, which then translates to $K(\pi v)=-(J^s)''(0)/J^s(0)=0$ from the Jacobi equation \eqref{eq: jacobi surface}.

The last statement is an easy consequence of the second statement and the alternative characterization \eqref{eq: char Sing by K} of the singular set.
\end{proof}

\begin{remark}\label{rem: difference in setting}
We note here that the behavior of $\lambda$ slightly differs between non-positively curved manifolds and manifolds with no focal points. For a non-positively curved manifold $M$, whenever $\lambda(v) = 0$ for some $v\in T^1M$, then it necessarily follows that $\lambda(g_tv)=0$ either for all $t\geq 0$ or for all $t\leq 0$. This is due to the convexity of the function $t \mapsto \|J(t)\|$ for any Jacobi field $J$. Such a monotonic behavior of $\lambda$ then translates to other geometric properties, including $d(\Sing,g_tv) \to 0$ as $t \to \infty$ or $t \to -\infty$; see \cite[Section 3]{burns2018unique}. 

The analogous property, however, does not hold for manifolds without focal points. Instead, a related function $\lambda_T$ serving a similar purpose was considered in \cite{chen2020unique,chen2021properties}, and we also make use of it in $\mathsection$ \ref{sec: proof}.
\end{remark}

For each $\eta>0$, we define
$$\Reg(\eta):=\{v\in T^1S \colon \lambda(v) \geq \eta\}.$$
While it defines a nested family of compact subsets in $\Reg$, it does not exhaust $\Reg$ as there are vectors $v\in \Reg$ with $\lambda(v) =0$. Nevertheless, the geodesic flow restricted to $\Reg(\eta)$ is uniformly hyperbolic; see Proposition \ref{prop: compact hyperbolic} or \cite[Lemma 3.10]{burns2018unique}. Using this fact together with the uniform continuity of $\lambda$,
we obtain the following version of the shadowing lemma for orbit segments whose endpoints lie in $\Reg(\eta)$. 
This lemma is similar to \cite[Lemma 2.4]{burns2014lyapunov} which establishes the shadowing lemma for orbit segments whose endpoints lie in subsets of the form $\{v\in T^1S : K( \pi v)\leq -\eta\}$. The proof there readily extends to orbit segments with endpoints in $\Reg(\eta)$, and we omit the proof.


\begin{lemma}[Shadowing lemma]\label{lem: shadowing} Let $S$ be a closed surface without focal points. 
For any $\eta,\ep,\tau>0$, there exists $\delta>0$ such that for any collection of orbit segments $\{(v_i,t_i)\}_{i \in \Z}$ with $v_i,g_{t_i}v_i\in \Reg(\eta),$ $t_i \geq \tau$ and $d(g_{t_i}v_i,v_{i+1})<\delta$ for all $i \in \Z$, there exist a geodesic $\gamma$ and a sequence of times $\{T_i\}_{i\in \Z}$ with $T_0 = 0$, $T_i+t_i-\ep \leq T_{i+1} \leq T_i+t_i+\ep$, and $d(\dot{\gamma}(t),\dot{\gamma}_{v_i}(t-T_i)) <\ep$ for all $t\in [T_i,T_{i+1}]$ and $i \in \Z$.

The geodesic $\gamma$ is unique upto re-parametrization. Moreover, if the orbits being shadowed are periodic, then the shadowing orbit is also periodic. 
\end{lemma}

Another useful structure coming from the foliations is the local product structure. In the following definition, $B(v,\d)$ denotes an open ball of radius $\d$ around $v$.
Since the foliations $W^\sigma$, $\sigma \in \{s,cs,u,cu\}$, are continuous, any compact subset of $\Reg$ has a uniform local product structure for some $\delta>0$ and $\kappa\geq 1$. 

\begin{definition}[Local product structure]\label{defn: LPS} 
We say the foliations $W^{cs}$ and $W^{u}$ have the \textit{local product structure}
at scale $\delta>0$ with constant $\kappa\geq1$ at $v\in T^1S$ if for any $w_{1},w_{2}\in B(v,\delta)$,
the intersection $[w_{1},w_{2}]:=W_{\kappa\delta}^{u}(w_{1})\cap W_{\kappa\delta}^{cs}(w_{2})$
is a unique point and satisfies 
\begin{align*}
d^{u}(w_{1},[w_{1},w_{2}]) & \leq\kappa d(w_{1},w_{2}),\\
d^{cs}(w_{2},[w_{1},w_{2}]) & \leq\kappa d(w_{1},w_{2}).
\end{align*}
\end{definition}

\subsection{Thermodynamic Formalism}
In this subsection, we briefly survey relevant results in thermodynamic formalism and multifractal analysis. 

A probability measure $\mu$ on $T^1S$ is \textit{$G$-invariant} if $\mu$ is $g_t$-invariant for all $t \in \R$, and we denote the set of all invariant measures by $\M(G)$. We say an $\mu\in \M(G)$ is \emph{hyperbolic} if its Lyapunov exponent (see Definition \ref{defn: Lyap exp}) is non-zero.
We denote by $h(\mu)$ the metric entropy of the time-one map of the flow $g_1$ with respect to $\mu$. For any $G$-invariant subset $Z \subseteq T^1S$, we denote the entropy of $g_1$ restricted $Z$ by $h(Z)$. If $Z$ is non-compact, such as the Lyapunov level sets $\L(\b)$, we adopt Bowen's definition \cite{bowen1973topological} of entropy for non-compact sets. 
 
For any continuous flow $G = \{g_t\}_{t \in \R}$ on a compact metric space $X$ and any continuous function (also known as the \emph{potential}) $\vp$ on $X$, the 
\emph{pressure} $P(\vp)$ of $\vp$ can be defined by the \emph{variational principle}:
\begin{equation}\label{eq: var prin}
P(\vp) = \sup \Big\{h(\mu)+ \int\vp \,d\mu \colon \mu \in \M(G)\Big\}.
\end{equation}
Any invariant measure, if any, achieving the supremum is called an \textit{equilibrium state}. We refer the reader to \cite{walters2000introduction} for more details.

From the entropy-expansiveness of the geodesic flows over manifolds with no focal points \cite{liu2016entropy}, the entropy function $\mu \mapsto h(\mu)$ is upper semi-continuous, and hence, there exists at least one equilibrium state for any potential.

For geodesic flows over rank 1 non-positively curved manifolds, Burns, Climenhaga, Fisher, and Thompson \cite{burns2018unique} showed that $t \vpgeo$ has a unique equilibrium state for $t$ in a small neighborhood of 0. In the case of surfaces, they also established the result for all $t<1$. In our setting of surfaces with no focal points, the analogous result is established by Chen, Kao, and the first named author:

\begin{proposition}\label{prop: CKP}\cite[Theorem C]{chen2020unique}
Let $S$ be a rank 1 Riemannian surface with no focal points. Then $t\vpgeo$ has a unique equilibrium state $\mu_t\in \M(G)$ for each $t < 1$.
\end{proposition}

Recall that the geometric potential is a continuous function on $T^1S$ defined by $$\vpgeo(v):=-\lim\limits_{t \to 0} \frac{1}{t}\log\Big\|dg_t|_{E^u_v}\Big\|.$$
In particular, $-\vpgeo(v)$ captures the instantaneous growth rate of $E^u_v$ under the derivative of the geodesic flow. 

\begin{definition}\label{defn: Lyap exp}
The \textit{forward Lyapunov exponent} $\chi^+(v)$ of $v$ is defined as
$$\chi^+(v):=\lim\limits_{t \to \infty}\frac{1}{t}\log \Big\|dg_t|_{E^u_v}\Big\|=\lim\limits_{t\to\infty}-\frac{1}{t} \int_0^t \vpgeo(g_sv)\,ds.$$
The \textit{backward Lyapunov exponent}
$$\chi^-(v):=\lim\limits_{t \to \infty}\frac{1}{-t}\log\Big \|dg_{-t}|_{E^u_v}\Big\| =\lim\limits_{t\to\infty}-\frac{1}{t} \int^0_{-t} \vpgeo(g_sv)\,ds.$$

We say $v \in T^1S$ is \textit{Lyapunov regular} if $\chi^+(v)=\chi^-(v)$, in which case we denote its common value by $\chi(v)$. The $\b$-\emph{Lyapunov level set} is defined as 
$$\L(\beta) := \{v\in T^1S \colon v \text{ is Lyapunov regular and }  \chi(v) =\beta\}.$$
For $\mu\in \M(G)$, we define its \emph{Lyapunov exponent} by $\displaystyle \chi(\mu):=-\int{\vpgeo \,d\mu}.$
\end{definition}
 
\begin{remark} \label{rem: Sing in L0}
Notice that the singular set is contained in the 0-level set $ \L(0)$.
This follows because for any $v \in \Sing$, its stable subspace $E^s_v$ coincides with its unstable subspace $E^u_v$, and hence, the unstable Jacobi field $J^u$ along $\gamma_v$ has constant length. Hence, $\vpgeo$ vanishes on the singular set, and we have $\chi(v) = 0$ for any $v \in \Sing$.
\end{remark}

As outlined in the introduction, we conduct our multifractal analysis on $\mathcal{L}(\beta)$ by studying the pressure function $\P$ and its Legendre transform $\E$ defined as in the \eqref{eq: P and E}:
$$
\mathcal{P}(t):=P(t\vpgeo) ~\text{ and }~  \mathcal{E}(\a):=\inf_{t\in \mathbb{R}}(\mathcal{P}(t)-t\a).
$$
The following proposition summarizes some useful properties of $\mathcal{P}$ in our setting:

\begin{proposition}\label{prop: BG thermo and multifractal}
With $\mathcal{P}$ as above, we have
\begin{enumerate}
\item $\mathcal{P}$ is non-increasing and convex, and $\mathcal{P}(t)=0$ for all $t\geq 1$.
\item $\mathcal{P}$ is $C^1$ everywhere except for $t=1$. Moreover, $\displaystyle \mathcal{P}'(t)=\int \varphi^{geo}d\mu_t$ for all $t<1$, where $\mu_t$ is the unique equilibrium state for $t\vpgeo$ from Proposition \ref{prop: CKP}.
\item Recalling the definition of $\alpha_1:=\lim\limits_{t\rightarrow -\infty}D^{+}\P(t)$ from the introduction, for every $\alpha\in [\alpha_1,0]$ there exists a unique supporting line $\ell_{\alpha}$ to $\mathcal{P}$ of slope $\alpha$. 
\item  For $t<1$, the unique supporting line to $\mathcal{P}$ at $(t,\mathcal{P}(t))$  is
$$\ell_{\alpha_t}(s):=h(\mu_t)+s\a_t,$$ 
where $\displaystyle\alpha_t:=\P'(t)=\int \varphi^{geo}d\mu_t$. In particular, $\mathcal{E}(\alpha_t)=h(\mu_{t}).$
\end{enumerate}
\end{proposition}
\begin{proof}[Proof sketch]
The first statement follows from the fact that the singular set is non-empty. The second statement is due to the variational principle \eqref{eq: var prin} and Proposition \ref{prop: CKP}; see \cite[Proposition 5]{burns2014lyapunov}. The third statement follows from the first statement, and the last statement follows from the second statement. 
\end{proof}

At $t=1$, there is a \emph{phase transition}, which is caused by the discrepancy between $$\alpha_2:=D^{-}\P(1) \text{ and } D^{+}\P(1)=0$$
where $\alpha_2$ is well-defined due to the convexity of $\mathcal{P}$. 
Due to the geometric nature of the setting, all $\alpha$ appearing in this paper, including $\a_1$ and $\a_2$, are non-positive.

As outlined in the introduction, the proof of Theorem \ref{thm: main} appearing in the next section considers two domains $(\a_1,\a_2)$ and $[\a_2,0]$ for $\a$ separately. In analyzing latter case, we will make use of uniformly hyperbolic subsystems, called the basic sets, whose Lyapunov level sets and their multifractal information are better understood.

\begin{definition}\label{defn: basic}
A \textit{basic set} $\Lambda \subset T^1S$ is a compact, $G$-invariant, and locally maximal hyperbolic set on which the geodesic flow is transitive.
\end{definition}

For any basic set $\Lambda \subset T^1S$, analogous to \eqref{eq: P and E} we define 
$$
\mathcal{P}_\Lambda(t):=P_\Lambda(t\vpgeo)
 \text{ and }
\mathcal{E}_\Lambda(\a):=\inf_{t\in \mathbb{R}}(\mathcal{P}_\Lambda(t)-t\a).
$$
Basic sets are useful in our analysis because these functions capture precise information of the level sets $\L(\b) \cap \Lambda$. 
We summarize them in the following proposition. Proofs for the first four properties can be found in \cite[Proposition 6]{burns2014lyapunov} and the last property is due to Barreira and Doutor \cite{BD04}.
	
\begin{proposition}\label{prop: basic} 
Let $\Lambda \subset T^1S$ be a basic set. Setting $\alpha_1(\Lambda):=\lim\limits_{t\to -\infty}\mathcal{P}'(t)$ and $\alpha_2(\Lambda):=\lim\limits_{t\to \infty}\mathcal{P}'(t)$, we have
\begin{enumerate}
\item  $\mathcal{P}_\Lambda(t)$ is strictly convex and real analytic on $\R$.
\item For each $\a\in [\a_1(\Lambda),\a_2(\Lambda)]$, $\mathcal{P}_\Lambda(t)$ has a unique supporting line $\ell_{\Lambda,\a}$ of slope $\a$, which intersects vertical axis at $(0,\mathcal{E}(\alpha))$. 
\item For $\a\in [\a_1(\Lambda),\a_2(\Lambda)]$, we have $\mathcal{L}(-\alpha)\cap \Lambda \neq \emptyset$. Otherwise, $\mathcal{L}(-\alpha)\cap \Lambda = \emptyset$.

\item For every $t$, there is a unique equilibrium state $\mu_t$ for $t\vpgeo|_{\Lambda}$. Moreover, it satisfies $\chi(\mu_t) = -\mathcal{P}_\Lambda\,'(t)$ and
$$\mathcal{E}_\Lambda(-\chi(\mu_t)) = h(\mu_t).$$
\item For every $\a \in (\a_1(\Lambda),\a_2(\Lambda))$, we have
$$\dim_H (\L(-\a) \cap \Lambda)=1+2 \cdot \frac{\mathcal{E}_\Lambda(\a)}{-\a}$$
and 
$$h(\L(-\a) \cap \Lambda) = \mathcal{E}_\Lambda(\a) = \max\{h(\mu)\colon \chi(\mu) = -\a \text{ and }\text{supp}(\mu) \subseteq \Lambda \}.$$
\end{enumerate} 
\end{proposition}

\section{Proof of Theorem \ref{thm: main}}\label{sec: 3}
Since we generally follow the strategies of \cite{burns2014lyapunov}, we will either omit or sketch the proofs whenever applicable and refer the reader there for details. 
\subsection{Few simple observations}
We establish some partial statements of Theorem \ref{thm: main}.
\begin{proposition}\label{prop: before phase transition} 
We have
\begin{enumerate}
    \item $\mathcal{L}(-\alpha)=\emptyset$ for every $\alpha<\alpha_{1}$ and every $\alpha>0$.
    \item For every $\alpha\in (\alpha_1,\alpha_2)$, the Lyapunov level set $\L(-\a)$ is non-empty and
    $$h(\L(-\alpha))=\mathcal{E}(\alpha).$$
        \item $\L(-\a_1)$ is non-empty. 
\end{enumerate}
\end{proposition}

\begin{proof}
The first statement, which is essentially a reformulation of the variational principle \eqref{eq: var prin}, follows from \cite[Proposition 5 (2)]{burns2014lyapunov}.
For the second statement, we first define 
$$\mathcal{L}^+(\b):=\{v \in T^1S\colon \chi^+(v) = \b\},$$
and denote by $h(\mathcal{L}^+(\beta))$ its topological entropy. Clearly we have $\mathcal{L}(\beta)\subset \mathcal{L}^+(\beta)$. By a routine adaption of the proof for \cite[Lemma 3.3.3]{CliThe} to the case of continuous flows over compact spaces, we can draw a conclusion similar to \cite[Theorem 3.1.1 (1)]{CliThe}: for every $t\in \mathbb{R}$, we have (with the convention of $h(\emptyset)=-\infty$)
$$
    \mathcal{P}(t)=\sup_{\alpha\in \mathbb{R}}(h(\mathcal{L}^+(-\alpha))+t\alpha).
$$
As an immediate consequence, we have
\begin{equation}\label{eq: E general}
    \mathcal{E}(\alpha)=\inf_{t\in \mathbb{R}}(\mathcal{P}(t)-t\a)\geq h(\mathcal{L}^+(-\alpha)) \geq h(\mathcal{L}(-\alpha)).
\end{equation}


For the reverse inequality, let $t_\a <1$ be a real number such that $\P'(t_\a) = \a$. We then have from Proposition \ref{prop: BG thermo and multifractal} (4) that
$$
\mathcal{E}(\alpha)=h(\mu_{t_{\alpha}})=\inf\{h(Z)\colon Z\subset T^1S, ~\mu_{t_{\alpha}}(Z)=1\}\leq h(\mathcal{L}(-\alpha)),
$$
as required.

The last statement follows from the fact that the Lyapunov exponent of an invariant measure $\chi(\mu)$ is defined as $-\int \vpgeo d\mu$. Indeed, taking any weak$^*$-limit $\mu \in \M(G)$ of the unique equilibrium state $\mu_{t_\a}$ for $t_\a\vpgeo$ as $\a \to \a_1$, we must have that $\chi(\mu) = -\a_1$. By applying the ergodic decomposition to $\mu$ and using the choice of $\alpha_1$, there exists an ergodic $\mu_e\in \mathcal{M}(G)$ such that $\chi(\mu_e) = -\a_1$. Then the set of generic points for $\mu_e$ is non-empty and belongs to $\mathcal{L}(-\alpha_1)$. 
\end{proof}

We conclude this subsection by noting that the above argument for \eqref{eq: E general} readily extends for any $\alpha\in [\alpha_1,0]$. Hence, for all such $\a$ we have 
$\mathcal{E}(\alpha)\geq h(\mathcal{L}(-\alpha))$.

\subsection{Remaining statements of Theorem \ref{thm: main}}
Throughout this subsection, we will assume the following proposition whose proof appears in $\mathsection$ \ref{sec: proof}.

\begin{proposition} \label{prop: Lambda sequence}
There exists an increasingly nested sequence of basic sets $\{\widetilde{\Lambda}_i\}_{i\in \mathbb{N}}$ such that for any basic set $\Lambda\subset T^1S$, there exists $n\in \mathbb{N}$ such that $\Lambda \subseteq \widetilde{\Lambda}_n$.
\end{proposition}

As we will see, such a family $\{\wt{\Lambda}_n\}_{n\in \N}$ of basic sets is used to extract information of the Lyapunov level sets. It is used in showing that $\L(-\a)$ is non-empty for $\a \in (\a_1,0)$. 
Moreover, it is also used in showing that $\mathcal{P}_{n}(t):=\mathcal{P}_{\wt{\Lambda}_n}(t)$ converges to $\mathcal{P}(t)$ for each $t$. This then implies that the supporting line $\ell^n_\a$ to $\mathcal{P}_n(t)$ of slope $\a$ converges to the supporting line $\ell_\a$ to $\mathcal{P}(t)$ of slope $\a$ for all $\alpha\in (\alpha_1,0)$. This observation then translates to the statements of Theorem \ref{thm: main}.

We begin by relating the Lypaunov exponent $\chi(v)$ of a Lyapunov regular vector $v\in T^1S$ to the solution of the Riccati equation over $\gamma_v$.
For any orthogonal Jacobi field $J$, it follows from the Jacobi equation \eqref{eq: jacobi surface} that $u:=J'/J$ satisfies the \textit{Riccati equation} given by
\begin{equation}\label{eq: riccati}
u'(t)+u(t)^2+K(\gamma(t)) = 0.
\end{equation}
Then the Lyapunov exponent $\chi(v)$ may be described by
\begin{equation}\label{eq: chi(v) formula}
\chi(v)=\lim_{T\rightarrow{\infty}}\frac{1}{T}\int_0^T u(t)\,dt,
\end{equation}
where $\displaystyle u(t):=J_{\xi}'(t)/J_{\xi}(t)$ with $\xi\in E^u_v$ is a solution to the Riccati equation \eqref{eq: riccati}; see \cite[Lemma 2.6]{burns2014lyapunov}. 
Such a description provides an upper bound for the Lyapunov exponent of a closed geodesic.

\begin{lemma}\label{lem: chi(v) periodic}\cite[Lemma 2.9]{burns2014lyapunov} For any closed geodesic $(v,t) \in T^1S \times [0,\infty)$, we have 
$$\chi(v)\leq \sqrt{-\frac{1}{t}\int_0^{t}K(\g_v(s))\,ds}.$$
\end{lemma}

The right hand side of the expression in the lemma above is well-defined 
because for such a closed geodesic we have $u(s)=u(s+t)$ for all $s\in \mathbb{R}$. This then implies that $\displaystyle \int_0^t u'(s)\,ds=0$. Plugging this into \eqref{eq: riccati} gives
$\displaystyle \int_0^t u^2(s)+K(\gamma_v(s))\,ds=0,
$
and hence
$$
    \int_0^t K(\gamma_v(s))ds=-\int_0^t u^2(s)\,ds\leq 0.
$$

Using  Lemma \ref{lem: shadowing} and \ref{lem: chi(v) periodic}, we obtain closed geodesics with arbitrarily small Lyapunov exponents:
\begin{proposition} \label{small Lya} 
If $\Sing$ is non-empty, then there exist closed geodesics with Lyapunov exponents arbitrarily close to 0 and $-\a_1$. 
\end{proposition}
\begin{proof} We treat two cases separately.
Let $\ep>0$ be arbitrary. 
Recalling the nested family of compact subsets $\Reg(\eta)$ defined below Remark \ref{rem: difference in setting}, Proposition \ref{prop: lambda properties} shows that their union $\bigcup\limits_{\eta>0} \Reg(\eta)$ contains the set $\{v \in T^1S \colon K(\pi v) \neq 0\}$.
In particular, there exists $\eta=\eta(\ep)>0$ such that $$\{v\in T^1S: |K(\pi v)| \geq \ep\}\subseteq \text{Reg}(\eta).$$

From the transitivity of the geodesic flow (Proposition \ref{prop: properties}), we can then choose a sequence of orbits segments $\{(v_n,t_n)\}_{n\in \N}$ lying entirely in the set $\{v\in T^1S: |K(\pi v)| \leq \ep\}$ such that the footprints of both their endpoints $v_n$ and $w_n:=g_{t_n}v_n$ have the Gaussian curvature equal to either $\ep$ or $-\ep$ and that $t_n \to \infty$. By passing to a subsequence if necessary, suppose that $v_n \to v$ and $w_n \to w$.

Again from the transitivity, we can find an orbit segment starting near $w$ and terminating near $v$. Then using the shadowing lemma (Lemma \ref{lem: shadowing}) applied to $\Reg(\eta/2)$, we obtain a closed geodesic which shadows $v_n$ to $w_n$ (for some sufficiently large $n\in \N$) followed by $w$ to $v$. This closed geodesic spends most of its time in $\{v\in T^1S: |K(\pi v)| \leq 2\ep\}$, and hence, its Lyapunov exponent can be bounded above by $2\sqrt{\ep}$ using Lemma \ref{lem: chi(v) periodic}. Since $\ep>0$ was arbitrary, this constructs closed geodesics whose Lyapunov exponents are arbitrarily close to 0. See \cite[Proposition 3]{burns2014lyapunov} for details.

For closed geodesics with Lyapunov exponents arbitrarily close to $-\a_1$, we fix any $v_0\in \mathcal{L}(-\alpha_1)$ and consider the solution $u(t)$ to the Riccati equation \eqref{eq: riccati} over $\gamma_{v_0}$. The existence of such $v_0\in T^1S$ is guaranteed from Proposition \ref{prop: before phase transition}.
Since $\displaystyle -\alpha_1=\chi(v_0)=\lim_{T\rightarrow{\infty}}\frac{1}{T}\int_0^T u(t)\,dt$, it is not hard to see that $\liminf\limits_{t\rightarrow \pm\infty}K(\gamma_{v_0}(t))<0$ by studying the evolution of $u(t)$. In particular, there exists $\ep>0$ such that $g_tv_0$ enters $\{v\in T^1S: K(\pi v) \leq -\ep\}$ infinitely often in both forward time at $\{t_n\}_{n\in \N} \subset \R_+$ and backward time $\{s_n\}_{n\leq 0} \subset \R_-$. Moreover, we have
\begin{equation}\label{eq: sntn}
\frac{1}{t_n-s_n} \int_{s_n}^{t_n}-\vpgeo (g_\tau v_0)\,d\tau \to -\a_1.
\end{equation}

Using \eqref{eq: sntn} instead of Lemma \ref{lem: chi(v) periodic}, the rest of the proof can be completed as above. By passing to a subsequence if necessary, we may assume that $g_{s_n}v_0 \to v$ and $g_{t_n}v_0 \to w$ and fix an orbit segment starting near $w$ and terminating near $v$. Then it follows from \eqref{eq: sntn} that the Lyapunov exponent of the closed geodesic obtained from shadowing $v_n$ to $w_n$ followed by $w$ to $v$ limits to $-\a_1$ as $n\to \infty$.
\end{proof}

Recall that $\ell_\a$ is the supporting line to $\mathcal{P}(t)$ of slope $\a$.
The following proposition shows that $\mathcal{P}_n(t)$ and $\ell^n_\a$ converge to $\mathcal{P}(t)$ and $\ell_\a$, respectively. Its proof makes uses of the Katok's horseshoe theorem (see \cite{katok1980lyapunov} and also \cite{gelfert2016horseshoes,burns2014lyapunov}) 	which states that for any $\ep>0$, hyperbolic ergodic measure $\mu \in \M(G)$, and potential $\vp \colon T^1S \to \R$, there exists a basic set $\Lambda \subseteq T^1S$ such that $$P_{\Lambda}(\varphi)>P_{\mu}(\varphi)-\ep$$
where $\displaystyle P_\mu(\vp):= h_{\mu}(G)+\int \vp \,d\mu$.

\begin{proposition}\label{prop: main Lambda convergence}
The sequence of basic sets $\{\wt{\Lambda}_n\}_{n\in \mathbb{N}}$ from Proposition \ref{prop: Lambda sequence} satisfies $$\mathcal{P}_n(t):=\mathcal{P}_{\wt{\Lambda}_n}(t)\nearrow\mathcal{P}(t)$$ pointwise. Moreover, $\ell^n_\a$ converges to $\ell_\a$ for each $\a \in (\a_1,0).$
\end{proposition}

\begin{proof}
Let $\ep>0$ be given. We divide the proof into two cases $t<1$ and $t \geq 1$. 

For $t <1$, let $\mu_t \in \M(G)$ be the unique equilibrium state of $t\vpgeo$ which is necessarily hyperbolic; see \cite{chen2020unique}. Applying Katok's horseshoe theorem to $\mu_t$ and $t\vpgeo$ produces a basic set $\Lambda$ satisfying
 $$P_{\Lambda}(t\vpgeo)>P_{\mu_t}(t\vpgeo)-\ep=\P(t)-\ep.$$ Then Proposition \ref{prop: Lambda sequence} implies that $\P_{n}(t) \geq \P(t)-\ep$ for all $n$ large enough. 

When $t\geq 1$, note that $\P(t) = 0$ from Proposition \ref{prop: BG thermo and multifractal}. Proposition \ref{small Lya} produces a hyperbolic measure $\mu_{\ep} \in \M(G)$ whose Lyapunov exponent $\chi(\mu_\ep)$ lies in the interval $(0,\ep/t)$. Since $\mu_\ep$ is supported on a closed geodesic, it has zero entropy, and hence $P_{\mu_{\ep}}(t\vpgeo)$ is equal to $-t\chi(\mu_\ep)$. By applying Katok's horseshoe theorem to $\mu_\ep$ and $t\vpgeo$, we obtain a basic set $\Lambda$ such that $$\P_{\Lambda}(t)>P_{\mu_{\ep}}(t\vpgeo)-\ep>-2\ep.$$ Proposition \ref{prop: Lambda sequence} then implies that $\P_{n}(t)>-2\ep$ for all sufficiently large $n$. Since $\ep$ was arbitrary, this completes the proof of the first statement.

The proof of the second statement can be found in \cite[Proposition 12]{burns2014lyapunov}.
\end{proof}

We will prove each statement of Theorem \ref{thm: main} separately.

\begin{proof}[Proof of the first statement of Theorem \ref{thm: main}]
From Proposition \ref{prop: before phase transition}, $\L(-\a)$ is empty if $\a \in (-\infty,\a_1)$ and $\a \in (0,\infty)$. From the same proposition, $\mathcal{L}(-\alpha)$ is non-empty for all $\alpha\in [\alpha_1,\alpha_2)$. Moreover, $\mathcal{L}(0)$ is non-empty as $\Sing\subset \mathcal{L}(0)$; see Remark \ref{rem: Sing in L0}. Therefore, it suffices to show that $\L(-\a)$ is non-empty for $\a \in [\a_2,0)$.

In fact, the following method is general enough that it shows the domain in which $\L(-\a)$ is non-empty contains a closed interval whose endpoints are arbitrarily close to $\a_1$ and $0$.
Using Proposition \ref{small Lya}, we begin by choosing closed geodesics with the Lyapunov exponents arbitrarily close to 0 and $-\a_1$, each of which is a basic set. Then the increasingly nested family $\{\wt{\Lambda}_n\}_{n\in \N}$ from Proposition \ref{prop: Lambda sequence} eventually contains both such closed geodesics for all sufficiently large $n\in \N$. The claim then follows from the fact that the set of $\b \in \R$ in which $\Lambda \cap \L(\b)$ is non-empty is a closed interval for any basic set $\Lambda$; see Proposition \ref{prop: basic}.
\end{proof} 

In view of Proposition \ref{prop: before phase transition} and the comment after its proof, the only thing left to show in Theorem \ref{thm: main} is the lower bound for the entropy and dimension spectrum. Again, the proof of the second statement of Theorem \ref{thm: main} relies on the increasingly sequence of basic sets $\{\widetilde{\Lambda}_n\}_{n\in \mathbb{N}}\subset \Reg$ constructed in Proposition \ref{prop: Lambda sequence}. Similar to \eqref{eq: P and E}, we define $$\mathcal{E}_{n}(\alpha)=\mathcal{E}_{\widetilde{\Lambda}_n}(\alpha):=\inf_{t\in \mathbb{R}}(\mathcal{P}_{n}(t)-t\alpha)$$ as the Legendre transform of $\mathcal{P}_n$ at $\alpha$.

\begin{proof}[Proof of the second statement of Theorem \ref{thm: main}]
From Proposition \ref{prop: main Lambda convergence}, the supporting lines $\ell^n_\a$ to $\mathcal{P}_n(t)$ converges to the supporting line $\ell_\a$ to $\mathcal{P}(t)$ for each $\a \in (\a_1,0)$. Since $\ell_\a$ intersects the vertical axis at $(0,\mathcal{E}(\a))$, and likewise for $\ell^n_\a$ at $(0,\mathcal{E}_n(\a)$), it follows that $\mathcal{E}_n(\a)$ also monotonically converges to $\mathcal{E}(\a)$ for all $\a \in (\a_1,0)$. Using Proposition \ref{prop: basic}, this then translates to 
$$h(\L(-\a))\geq \lim_{n\rightarrow \infty}h(\L(-\alpha) \cap \wt{\Lambda}_n)=\lim_{n\rightarrow \infty}\mathcal{E}_n(\alpha)=\mathcal{E}(\alpha)$$ and $$\dim_H(\L(-\a))\geq \lim_{n\rightarrow \infty}\dim_H(\L(-\a) \cap \wt{\Lambda}_n)= \lim_{n\rightarrow \infty}1+2\cdot \frac{\mathcal{E}_{n}(\alpha)}{-\alpha}=1+2\cdot \frac{\mathcal{E}(\alpha)}{-\alpha},$$ resulting in the required lower bounds for $h(\L(-\a))$ and $\dim_H(\L(-\a))$.
\end{proof}

\section{Proof of Proposition \ref{prop: Lambda sequence}}\label{sec: proof}

This section is devoted to the proof of Proposition \ref{prop: Lambda sequence}. In doing so, we will make use of a family of non-negative functions $\lt$ defined for each $T>0$. It first appeared in \cite{chen2020unique} and is closely related to $\lambda$:
$$
\lt(v) := \int^T_{-T} \lambda(g_tv)dt. 
$$
From Proposition \ref{prop: lambda properties}, if $\lt(v)=0$ for all $T>0$, then $v \in \Sing$. In particular, setting 
$$\Reg_T(\eta):=\{v \in T^1S \colon \lt(v) \geq \eta\},$$
we have
$$
\Reg = \bigcup\limits_{T,\eta>0} \Reg_T(\eta).
$$
Compared to $\Reg(\eta)$ which cannot exhaust the regular set, this has an advantage of being able to find $T,\eta>0$ such that $\Reg_T(\eta)$ contains a given compact subset of the regular set.

\subsection{Criteria for hyperbolicity}
Let 
$k_{\max} := \max\limits_{v\in T^1S}k^u(v)=\max\limits_{v\in T^1S}k^s(v)$
be the maximum geodesic curvature on $T^1S$.
It can be used in comparing $\lambda$ and $\lambda_T$ (see \cite[Subsection 4.2]{chen2020unique}): for any $v\in T^1S$ and $t >0$, we have
$$
\int_0^t \lambda(g_s v)\,ds \geq \frac{1}{2T}\int_0^t \lambda_T(g_s v)\,ds-2T k_{\text{max}}
$$
This inequality can be used in the following lemma to establish the uniform hyperbolicity on $\Reg_T(\eta)$ with the constant $C:=\exp(2Tk_{\max})$. We omit the proof as it is an easy generalization of \cite[Lemma 3.1]{burns2018unique} and \cite[Lemma 4.5]{chen2020unique}.
\begin{lemma}\label{lem: dist formula}
For any $T,\eta>0$, there exists $C>0$ such that the following holds: for any $v\in T^1S$, $w,w' \in W^s(v)$, and $t \geq 0$, suppose that the segment connecting $w$ and $w'$ along $W^s(v)$ remains in $\Reg_T(\eta)$ under the action of $g_\tau$ for all $\tau \in [0,t]$.
Then
$$d^s(g_tw,g_tw') \leq  C\cdot d^s(w,w') \cdot \exp\Big(-\frac{\eta t}{2T}\Big).$$
Similarly, for any $v\in T^1S$, $w,w' \in W^u(v)$, and $t \geq 0$ such that the segment connecting $w$ and $w'$ along $W^u(v)$ remains in $\Reg_T(\eta)$ under the action of $g_{-\tau}$ for all $\tau \in [0,t]$, 
$$d^u(g_{-t}w,g_{-t}w') \leq  C\cdot d^u(w,w') \cdot \exp\Big(-\frac{\eta t}{2T}\Big).$$
\end{lemma}

Using this lemma, the following proposition characterizes the uniformly hyperbolicity among compact subsets of $T^1S$.

\begin{proposition}\label{prop: compact hyperbolic}
Any compact $G$-invariant subset $\Lambda \subseteq T^1S$ is uniformly hyperbolic if and only if $\Lambda \subseteq \Reg$.
\end{proposition}
\begin{proof}
One direction is clear. If $\Lambda \cap \Sing \neq \emptyset$, then for any $v \in \Lambda \cap \Sing$, the geodesic $\gamma_v$ admits a Jacobi field $J_\xi$ with $\xi$ belonging to $ E^s(v) = E^u(v)$. From Proposition \ref{prop: properties}, such a Jacobi field $J_\xi$ has constant length for all $t\in \R$. This implies that $\Lambda$ is not uniformly hyperbolic.

Conversely, suppose that $\Lambda$ is a compact subset of $\Reg$. Then, there exist $T,\eta>0$ such that $\Lambda$ is contained in $\Reg_T(\eta)$. Then it follows from Lemma \ref{lem: dist formula} that $\Lambda$ is uniformly hyperbolic with rate $\eta/2T$.
\end{proof}

\begin{remark}\label{rem: top 1}
In addition to being uniformly hyperbolic, any compact $G$-invariant subset $\Lambda \subsetneq \Reg$ has topological dimension 1. This follows from \cite[Proposition 7]{burns2014lyapunov}
\end{remark}

\subsection{Concluding the proof of Proposition \ref{prop: Lambda sequence}}

The construction of the sequence of increasingly nested basic sets $\{\wt{\Lambda}_n\}_{n\in \N}$ from Proposition \ref{prop: Lambda sequence} begins with the following key proposition. 
This proposition corresponds to \cite[Proposition 8]{burns2014lyapunov}, and we prove it in the next section. While we follow their general strategy, due to limitations of $\lambda$ in characterizing the singular set (see Remark \ref{rem: difference in setting}), our construction relies on the function $\lambda_T$ more suited to our setting.  
\begin{proposition} \label{prop: key}
For any closed, $G$-invariant, hyperbolic set $\Lambda\subset T^1S$ and any neighborhood $U$ of $\Lambda$, there exists a closed, $G$-invariant, locally maximal, hyperbolic set $\widetilde{\Lambda}$ such that $\Lambda\subseteq \widetilde{\Lambda} \subseteq U$. 
\end{proposition}

Using Proposition \ref{prop: key}, the rest of the proof can be completed as in the proof of \cite[Theorem 1.4]{burns2014lyapunov}, which we briefly sketch here.	
The following lemma from \cite[Proposition 9]{burns2014lyapunov} allows to glue two basic sets into another basic set; it uses the shadowing lemma and its proof readily extends to our setting using our analogous version of the shadowing lemma (Lemma \ref{lem: shadowing}). 

\begin{lemma} \cite[Proposition 9]{burns2014lyapunov}   \label{lem: bridging}
Given any two basic sets $\Lambda_1,\Lambda_2 \subseteq T^1S$, there is a third basic set $\Lambda$ which contains them both.
\end{lemma}

For each $n\in \N$, we denote by $\wh{\Lambda}_n$ the closure of the set of all closed geodesics contained in the complement of the $\frac{1}{n}$-neighborhood $N_{\frac{1}{n}}(\Sing)$ of the singular set. It is a closed $G$-invariant hyperbolic set (Proposition \ref{prop: compact hyperbolic}). Then Proposition \ref{prop: key} creates a closed $G$-invariant hyperbolic set containing $\wh{\Lambda}_n$ with the added structure of the local maximality. The non-wandering set of such a set, which necessarily contains $\wh{\Lambda}_n$, can be decomposed into basic sets by Smale's spectral decomposition theorem, and repeatedly applying Lemma \ref{lem: bridging} produces a basic set $\wt{\Lambda}_n$ containing all such basic sets. In particular, $\wt{\Lambda}_n$ contains $\wh{\Lambda}_n$. By applying Lemma \ref{lem: bridging} again to $\wt{\Lambda}_{n-1}$, if necessary, we may assume that $\wt{\Lambda}_n$ contains $\wt{\Lambda}_{n-1}$ also.

It can be easily checked that such a nested sequence $\{\wt{\Lambda}_n\}_{n\in \N}$ is the required sequence basic sets. In fact, for any basic set $\Lambda$, the closed geodesics in the regular set are dense in $\Lambda$, and hence, there exists some $n\in\N$ such that $\Lambda \subset \wh{\Lambda}_n \subset \wt{\Lambda}_n$. This completes the proof of Proposition \ref{prop: Lambda sequence}.

\section{Proof of Proposition \ref{prop: key}}
\label{sec: 5}

\subsection{Choice of constants and preliminary construction}\label{sec: t}
 Let $\Lambda$ and $U$ be given as in Proposition \ref{prop: key}. We may assume that $U$ is sufficiently close to $\Lambda$ so that every $G$-invariant compact set contained in $U$ is hyperbolic and that $\overline{U}$ does not intersect the singular set. By Proposition \ref{prop: compact hyperbolic}, there exist $\eta,T>0$ such that $\overline{U}\subseteq \text{Reg}_{T}(\eta)$.
 
Since a compact hyperbolic set is locally maximal if and only if it has a local product structure, we will construct $\wt{\Lambda}$ as the image of a subshift of finite type under a continuous injective map. In this way, $\wt{\Lambda}$ inherits the natural local product structure from the shift space and will be locally maximal. 

Instead of building $\wt{\Lambda}$ directly, we will construct its intersection with a suitably chosen cross section $\C$ to the flow. This is because $\Lambda$ has topological dimension 1 from Remark \ref{rem: top 1}, so its intersection with a cross section has topological dimension 0 (i.e., discrete points). This helps with the construction of the Markov partition around $\Lambda \cap \C$, which will later be used to define the alphabets in the required subshift.

We first begin with the construction of the cross section $\C$. Denoting by $V \subset TT^1S$ the codimension-1 orthogonal complement to the vector field generating the geodesic flow, for any $v\in T^1S$ and $\d>0$ we define $D_{\delta}(v)$ as the image of the exponential map of the $\d$-ball in $V(v)$ centered at the origin.
For any $\delta$ sufficiently small and $v\in T^1S$, the center-stable foliation $W^{cs}$ induces an one-dimensional stable foliation $\W^s$ on $D_{\delta}(v)$ whose leaves are defined by 
$$\W^s_{v,\d}(w):= W^{cs}(w) \cap D_\d(v).$$
Likewise, $W^{cu}$ induces an one-dimensional foliation $\W^u_{v,\d}$ on $D_\d(v)$. When the underlying $D_\delta(v)$ is clear, we will often suppress $v$ and simply write $\W^{s/u}(w)$. Restricted to a compact subset of $\Reg$, the foliations $\W^{s/u}$ are uniformly transverse.

We will now choose the scale $\d>0$ to work with.
First, since $\Reg_{T}(\eta)$ is a compact subset of $\Reg$, there exist $\d_{LPS}=\delta(T,\eta)>0$ and $\kappa:=\kappa(T,\eta)>1$ such that vectors in a $\delta_{LPS}$-neighborhood of any $v\in \Reg_{T}(\eta)$ have the local product structure with constant $\kappa$; see Definition \ref{defn: LPS}. 

Since $E^s$ and $E^u$ are perpendicular to $E^c$, for sufficiently small $\d>0$, both $D_\d(v)$ and the foliations $W^{s/u}$ and $\W^{s/u}$ are nearly perpendicular to the geodesic flow lines. The slight failure of the orthogonality can be accounted for as follows. Once and for all, we fix $\b$ larger than but sufficiently close to $1$, and then we choose $\delta\in (0,\d_{LPS}/2)$ such that for any $v\in T^1S$,  $w\in D_{\delta}(v)$, and $u\in \W^{u}_v(w)$, we have 
\begin{equation}\label{eq: d and d^s}
\k^{-1} d^{u}(g_tu,w)\leq d(u,w)\leq \beta d^{u}(g_tu,w)
\end{equation}
where $t\in \R$ is a unique real number such that $g_tu$ belongs to the local unstable leaf $W^u(w)$ of $w$. The metric $d$ here denotes the metric from \eqref{eq: Knieper}, and the multiplicative factor $\k^{-1}$ shows up in the lower bound because $g_tu$ is equal to $[w,u]$, and the bounds from Definition \ref{defn: LPS} apply here.
Likewise, we may assume that analogous properties hold for any $u\in \W^{s}_v(w)$ with respect to $d^s$. Such $\delta>0$ fixed here will be the scale we will work with throughout the proof.


For $v\in \Lambda$, we say a subset of $D_\d(v)$ is a \emph{su-rectangle} if it is a rectangle where each edge is a subset of $W^{cs/cu}(w) \cap D_\d(v)$ for some $w \in \Sing$.  
Since both foliations $W^s$ and $W^u$ are minimal under the action of geodesic flow \cite[Lemma 4.6]{chen2021properties} and $\text{Sing}\neq \emptyset$, for each $v\in \Lambda$ we can build a su-rectangle around $v$ of arbitrarily small diameter contained in $D_{\delta}(v)$. 
In particular, we can choose a finite subset $\{v_i\}_{i=1}^n\subset \Lambda$ and build a sufficiently small su-rectangle $C_i:=C_{v_i}\subset D_\d(v_i)$ around $v_i$ such that the union $\displaystyle \bigcup_{1\leq i \leq n}g_{[-\frac{1}{2},\frac{1}{2}]}C_{v_i}$ contains $\Lambda$ and is contained in $U$. Moreover, we may construct $C_i$'s such that they are pairwise disjoint.
Our desired cross section is $\displaystyle \C:=\bigcup_{1\leq i \leq n}C_i$.

Since each $v_i$ belongs to $ \Lambda$, the induced foliations $\W^{s/u}$ are uniformly transverse on $\C$, and each $C_i$ has a natural product structure: for any $v_1,v_2\in C_i$, both
$$[v_1,v_2]_\C:=\W^u_{v,\delta}(v_1)\cap \W^s_{v,\delta}(v_2)$$ and $[v_2,v_1]_\C=\W^u_{v,\delta}(v_2)\cap \W^s_{v,\delta}(v_1)$ are contained in $C_i$.
For such $v_1,v_2$, there exists a unique real number $t\in \R$ such that $g_t[v_1,v_2]_\C$ coincides with $[v_1,v_2]$; see Figure \ref{fig: 1}. Note that it plays the same role as the $t$ appearing in \eqref{eq: d and d^s}. From the choice of $\b$ and $\k$, we have 
\begin{equation}\label{eq: d relation}
d(v_1,[v_1,v_2]_\C) \leq \b d^u(v_1,[v_1,v_2]) \leq \b \k d(v_1,v_2).
\end{equation}

Denote by $\partial C_i$ the boundary of $C_i$ for $1\leq i \leq n$ and $\partial \C$ the union of $\partial C_i$.
Note that for any $v_0\in \text{Sing}$ and $w_0\in W^{cs}(v_0)$, their forward Lyapunov exponents satisfy $\chi^+(w_0)=\chi^+(v_0)=0$. This follows from \cite[Proposition 2]{burns2014lyapunov} which makes use of the flat strip theorem; since the flat strip theorem remains valid for manifolds with no focal points  \cite[Theorem 2]{eschenburg1977horospheres}, the proof readily extends to our setting. Likewise, if $w_0\in W^{cu}(v_0)$, then $\chi^-(w_0)=\chi^-(v_0)=0$.
Since $\Lambda$ is uniformly hyperbolic, each $v\in \Lambda$ has $\chi(v)>0$. Hence, $\Lambda$ does not intersect the boundary of any su-rectangle. In particular, we have $\Lambda \cap \partial \C = \emptyset$.

\begin{figure}[h]
  \centering
  \begin{minipage}[b]{0.39\textwidth}\label{fig: 1}
    \includegraphics[width=\textwidth]{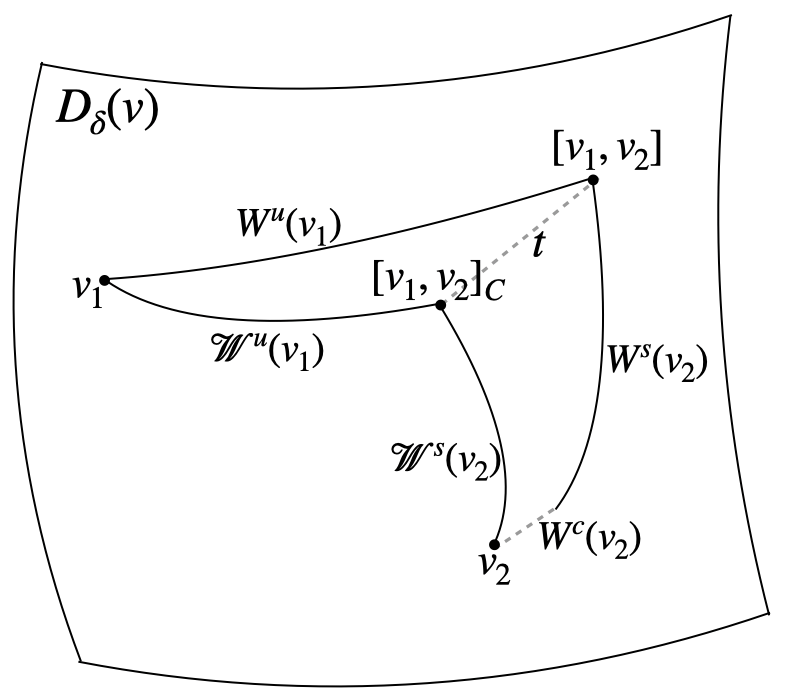}
    \caption{}
  \end{minipage}
  \begin{minipage}[b]{0.6\textwidth}\label{fig: 2}
    \includegraphics[width=\textwidth]{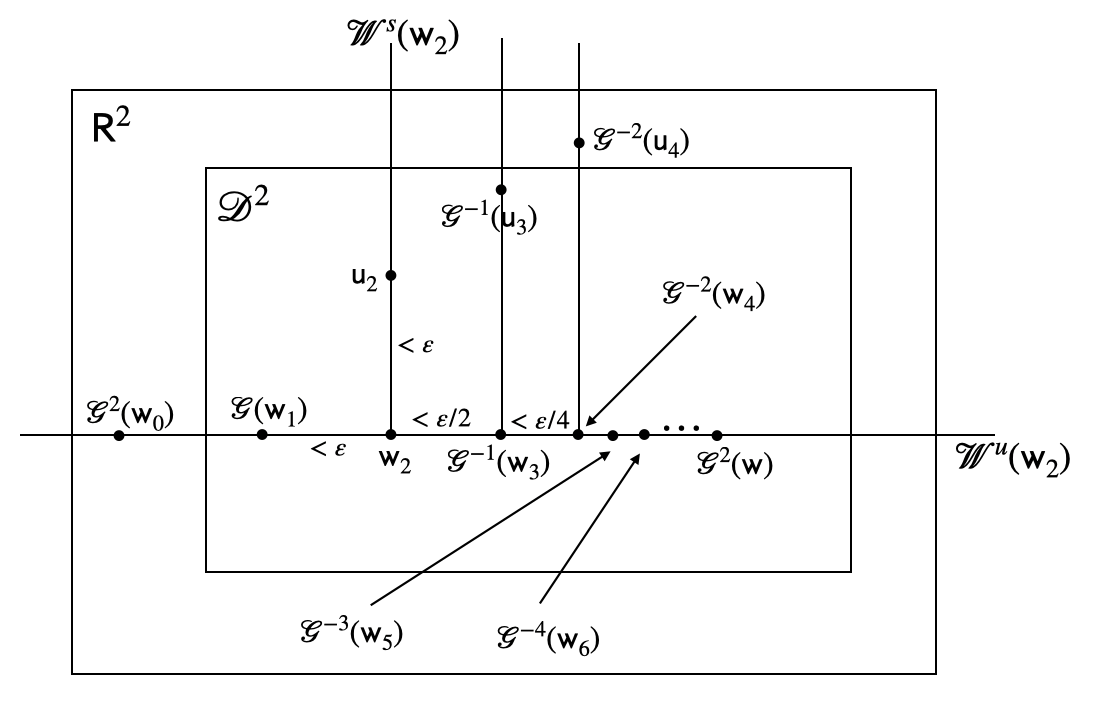}
    \caption{}
  \end{minipage}
\end{figure}
We fix a sufficiently small $0 \leq \a \ll \d$ and construct finitely many su-rectangles 
$$\mathcal{R}=\mathcal{R}(\alpha):=\{R_1,R_2,\cdots,R_m\}$$
 on $\C$ such that their union contains $\Lambda \cap \C$ and that each $R_i$ has a diameter at most $\a$ and contains at least one vector in $\Lambda$. We may assume that they have pairwise disjoint interiors and that the distance between $\Lambda \cap \C$ and $\bigcup\limits_{j=1}^m \partial R_j$ is positive using the same reasoning from the previous paragraph.

 By choosing $\a$ sufficiently small, we can ensure that the every vector in the union $\bigcup\limits_{j=1}^m g_{[0,1]}R_j$ is sufficiently close to $\Lambda$, and hence, is contained in $g_{[-\frac{1}{2},\frac{1}{2}]}\C$. In particular, the first return time $\tau \colon \bigcup\limits_{j=1}^m R_j \to [0,1]$ to $\C$ is well-defined, and we define accordingly the first return map by 
 $$\mathcal{F}(v):=g_{\tau(v)}v.$$ 
Since $g_{[-\frac{1}{2},\frac{1}{2}]}\C $ is contained in $U$ which is then contained in $\Reg_T(\eta)$, the uniform hyperbolicity established in Lemma \ref{lem: dist formula} applies to $\F$; see Lemma \ref{lem: hyperbolicity}. 

\begin{remark}\label{rem: F}
While $\F(v)$ necessarily belongs to $\C$, it may not belong to $\bigcup\limits_{j=1}^m R_j$. In particular, $\F$ may not be infinitely iterated on $\bigcup\limits_{j=1}^m R_j$. However, we will only need to consider and focus on the set of vectors on which $\F$ can be infinitely iterated.
\end{remark}

Since $C_i$'s are pairwise disjoint, there exists $c_1>0$ such that $\tau$ is bounded from below by $c_1$. 
Using $d(\Lambda,\partial \C)>0$, we may assume that $\mathcal{F}$ and $\tau$ are smooth on $R_j$ for every $1\leq j \leq m$.
To sum up, we have
\begin{enumerate}
    \item $\mathcal{R}=\{R_j\}_{j=1}^m$ is a collection of closed su-rectangles in $\C$ with mutually disjoint interiors. 
    \item The union $\displaystyle \bigcup_{j=1}^m R_j$ contains $\displaystyle \Lambda\cap \C$ and is contained in the interior of $\C$.
    \item Each $R_j$ has diameter at most $\alpha$ and contains at least one vector in $\Lambda\cap \C$.
    \item Both $\mathcal{F}$ and $\tau$ are smooth on $R_j$, and $\mathcal{F}(R_j)$ is contained in a single $C_i$ for every $1\leq j \leq m$.
\end{enumerate}

The following lemma shows that $\F$ preserves the local product structure on $\mathcal{R}$.

\begin{lemma}  \label{lemmalpsforR}
For any $v,w\in R_{i}$ such that $\mathcal{F}(v),\mathcal{F}(w)\in R_{j}$ for some $1 \leq i,j\leq m$, then
$$
    \mathcal{F}([v,w]_\C)=[\mathcal{F}(v),\mathcal{F}(w)]_\C.
$$

Similarly, if $v,w\in R_{i}$ and $\mathcal{F}^{-1}(v),\mathcal{F}^{-1}(w)\in R_{j}$, then
$$    \mathcal{F}^{-1}([v,w]_\C)=[\mathcal{F}^{-1}(v),\mathcal{F}^{-1}(w)]_\C.
$$
\end{lemma}
\begin{proof} [Proof of Lemma \ref{lemmalpsforR}]
Since two statements are symmetric, we will only show the first statement. 
Let $t_0,t_1$ be the unique real numbers such that $[v,w] =g_{t_0}[v,w]_\C\in W^u(v)$ and $[\F(v),\F(w)] =g_{t_1}[\F(v),\F(w)]_\C \in W^u(\F(v))$.

Note that $g_{\tau(v)+t_0}[v,w]_\C$ is the unique vector which lies on both $W^u(\F(v))$ and $W^{cs}(g_{\tau(v)}w)$. Since $W^{cs}(g_{\tau(v)}w)$ coincides with $W^{cs}(\F(w))$, we have $g_{\tau(v)+t_0}[v,w]_\C = g_{t_1}[\F(v),\F(w)]_\C$. In particular,  $\F$ maps $[v,w]_\C$ to $[\F(v),\F(w)]_\C$ with $\tau([v,w]_\C)=\tau(v)+t_0-t_1$.
\end{proof}

We use the elements from $\mathcal{R}$ to establish the alphabet in the target shift space. Following \cite{burns2014lyapunov}, for $N\geq 1$ we define 
$$\mathcal{R}_N :=\Big\{\mathcal{D} = \bigcap_{j=-N}^{j=N}\mathcal{F}^{-j}R^j \colon R^j\in \mathcal{R} \text{ and } \mathcal{D} \cap \Lambda \neq \emptyset\Big\}$$ as the collection of sets of the form $\displaystyle \bigcap_{j=-N}^{j=N}\mathcal{F}^{-j}R^j $ that contains at least one vector in $\Lambda$.

By translating Lemma \ref{lem: dist formula} to this setting, the following lemma establishes the uniform hyperbolicity of $\F$ restricted to vectors that belong to the same element of $\mathcal{R}$ under $\F$. 
\begin{lemma}\label{lem: hyperbolicity}
There exists $\g, C>0$ such that the following holds: for any $n\in \N$, $v \in T^1S$, and $w \in \W^s(v)$ such that $\F^j(v)$ and $\F^j(w)$ belong to the same element of $\mathcal{R}$ for every $0 \leq j \leq n$,
$$d(\F^n(v),\F^n(w)) \leq C e^{-\g n}d(u,v).$$
Similarly, for any $v \in T^1S$ and $w \in \W^u(v)$ such that $\F^j(v)$ and $\F^j(w)$ belong to the same element of $\mathcal{R}$ for every $-n \leq j \leq 0$,
$$d(\F^{-n}(v),\F^{-n}(w)) \leq C e^{-\g n}d(u,v).$$
\end{lemma}
\begin{proof}We prove the first statement; the second statement can be proved analogously. 
As above, there exist small $t_0,t_{1}\in \R$ such that $g_{t_{0}}w\in W^s(v)$ and that $g_{t_{1}}\F^{n}(w)\in W^s(\F^{n}(v))$. In particular, the line segment connecting $v$ and $g_{t_0}w$ along $W^s(v)$ is mapped under the time $S_n\tau(v):=\sum\limits_{j=0}^{n-1}\tau(\mathcal{F}^j(v))$ map of the geodesic flow to the line segment connecting $\F^{n}(v)$ and $g_{t_{1}}\mathcal{F}^{n}(w)$ along $W^s(\F^{n}(v))$. Since the entire process lies completely within $\Reg_T(\eta)$ from the construction, Lemma \ref{lem: dist formula} applies. 

Denoting the constant from Lemma \ref{lem: dist formula} by $C_0$, we have 
\begin{align*}
 d^s(\mathcal{F}^{n}(v),g_{t_{1}}\mathcal{F}^{n}(w)) &\leq C_0\exp\Big( -\frac{\eta}{2T} \cdot S^n_\tau(v) \Big) d^s(v,g_{t_0}w)\\
    &\leq C_0\k\exp\Big(-\frac{\eta}{2T} \cdot c_1 n \Big) d(v,w)
\end{align*}
where the second inequality uses \eqref{eq: d and d^s} as well as $S^n_\tau(v) \geq nc_1$ from the fact that the first return time $\tau$ is bounded below by $c_1$. Applying \eqref{eq: d and d^s} again to $ d^s(\mathcal{F}^{n}(v),g_{t_1}\mathcal{F}^{n}(w))$ establishes the lemma with $C :=C_0 \b\k$ and $\g := \eta c_1/(2T)$.
\end{proof}

The following lemma is an easy consequence of Lemma \ref{lem: hyperbolicity}.

\begin{lemma} \label{lemmasizeofRN}
For any $\ep>0$, there exists $N_2=N_2(\ep)\in \mathbb{N}$ such that $\text{diam}(\D)<\ep$ for every $N >N_2$ and $\D\in \mathcal{R}_N$. 
\end{lemma}
\begin{proof} [Proof of Lemma \ref{lemmasizeofRN}]
Let $\ep>0$ be given. In order to prove the lemma, it suffices to show that there exists $N_2 \in \N$ such that for any $N>N_2$, $\D\in \mathcal{R}_N$, $v\in \D\cap \Lambda$, and $w\in \D$, we have $\displaystyle d(v,w)<\ep/2$. 
 
We claim that we only need to show the cases where $w$ is on $\W^{s}(v) \cap \D$ or $\W^{u}(v) \cap \D$ with the upper bound $ \ep/2$ replaced by $\ep/(4\kappa\b+2)$. 
Indeed, suppose that $d(v,u)< \ep/(4\k\b+2)$ for any $u\in \W^{s}(v) \cap \D$ or $\W^{u}(v) \cap \D$. Then for any $w \in \D$, both $[v,w]_\C$ and $[w,v]_\C$ belong to $\D$ from Lemma \ref{lemmalpsforR}, and the triangle inequality gives $d([v,w]_\C,[w,v]_\C)<\ep/(2\k\b+1)$. Since $w$ coincides with $[[w,v]_C,[v,w]_\C]_\C$, we obtain $\displaystyle d([w,v]_\C,w)<\frac{\k\b \ep }{2\k\b+1}$ from \eqref{eq: d relation}. Combined with $\displaystyle d(v,[w,v]_\C)<\ep/(4\k\b+2)$, we get
$$d(v,w) \leq d(v,[w,v]_\C)+d([w,v]_\C,w) \leq \frac{\k\b\ep}{2\k\b+1}+\frac{\ep}{4\k\b+2}  = \frac{\ep}{2}.$$

Since the diameter of $R_i$ is bounded above by $\a$, the rest of the proof is now due to Lemma \ref{lem: hyperbolicity}. 
\end{proof}

While Lemma \ref{lem: hyperbolicity} establishes uniform hyperbolicity of $\F$, the following observation provides an alternative way to compare distance under small number of iterations of $\F$.
Consider any $u,v$ in the same element of $\C$ such that $u \in \W^s(v)$, and suppose that $\F^n(v)$ and $\F^n(u)$ belong to the same element of $\C$ for some $n\in \N$. There exist unique real numbers $t_0,t_1$ such that $g_{t_0}u\in  W^s(v)$ and $g_{t_1}\F^n(u) \in W^s(\F^n(v))$. Then $v$ and $g_{t_0}u$ are mapped to $\F^n(v)$ and $g_{t_1}\F^n(u)$ respectively under the time $S_n\tau(v)$ map of the geodesic flow. In particular,
using \eqref{eq: d and d^s} we have
\begin{equation}\label{eq: almost non-increase}
d(\F^n(v),\F^n(u)) \leq \b d^s(\F^n(v),g_{t_1}\F^n(u)) \leq \b d^s(v,g_{t_0}u) \leq \b\k d(v,u)
\end{equation}
where the second inequality is due to the fact that the $d^s$-distance is non-increasing with respect to the forward geodesic flow. Such an observation may be interpreted as that the distance between any two vectors on the same $\W^s$ leaf is almost non-increasing under $\F$. 
Likewise, the analogous statement holds for $d^u$ with respect to $\F^{-1}$.
\subsection{Choice of alphabets for the subshift}
We will choose large $N$ and use elements in $\mathcal{R}_N$ as the alphabet of the shift space. 
\begin{definition}
Given a bi-infinite sequence $(\cdots,a_{-1},a_0,a_1,\cdots)=(a_i)_{i\in \Z}$ with $a_i\in \mathcal{R}_N$ for all $i\in \mathbb{Z}$, we follow the definition from \cite{burns2014lyapunov} and call it \emph{$N$-admissible} if for any $i\in \mathbb{Z}$, there exists $u_i\in a_i\cap \Lambda$ such that $$\mathcal{F}(u_i)\in a_{i+1}.$$
\end{definition}
Denote by $\mathcal{A}_N$ the set of all $N$-admissible sequences. Notice that the $\mathcal{A}_N$ naturally has the local product structure defined by $$[a,b] := (\ldots,a_{-2},a_{-1},a_0,b_1,b_2,\ldots)$$
for any $a = (a_i)_{i\in \Z}$ and $b =(b_i)_{i\in \Z}$ with $a_0 = b_0$, and such a product structure will translate to the desired product structure on $\wt{\Lambda}$.

\begin{definition}
For any $a=(a_i)_{i\in \Z}\in \mathcal{A}_N$ and $\ep>0$, we call $w\in \C$ an \emph{$\ep$-shadowing} of $a$ if there exists $u_i\in \Lambda\cap a_i \cap \mathcal{F}^{-1}(a_{i+1})$ for each $i \in \Z$ such that
$$d(\mathcal{F}^i(w),u_i)<\ep.$$
\end{definition}

We will show in the next subsection that for any $\ep>0$ sufficiently small, there exists $N_0\in \N$ such that for any $N>N_0$, every element in $\mathcal{A}_N$ has a unique $\ep$-shadowing. Furthermore, we will show that such a shadowing map $\psi:\mathcal{A}_N\to \C$ is injective. 
To construct $\psi$, we begin by setting 
$$\Delta:=d\Big(\Lambda\cap \C, \bigcup\limits_{j=1}^m \partial R_j\Big),$$
 which is necessarily positive because $\mathcal{R}$ consists of su-rectangles and $\Lambda$ does not intersect the boundary of any su-rectangle. Fix $N_1\in \N$ such that $Ce^{-\g N_1}<1/2$ where $C,\g$ are from Lemma \ref{lem: hyperbolicity}, and choose 
 $$\ep\in\Big(0, \frac{\Delta}{2(1+\b^2\k^2)}\Big) \text{ and } N_0>\max\{N_1,N_2(\ep)\}$$
where $N_2$ is from Lemma \ref{lemmasizeofRN}.

We briefly summarize the consequences of such choices of constants.
With such a choice of $\Delta$, whenever $u \in \Lambda \cap \C$ and $v \in \C$ satisfies $d(u,v) <\Delta$, then $v$ belongs to the same element of $\mathcal{R}$ that contains $u$. From the choice of $N_0$,
whenever $w\in \W^s(u)$ such that $\F^i(w)$ and $\F^i(u)$ belong to the same element of $\mathcal{R}$ for all $0 \leq i \leq N_0$, then
$$
d(\mathcal{F}^{N_0}(u),\mathcal{F}^{N_0}(w))\leq \frac{1}{2}d(u,w).
$$
and likewise for $w \in \W^u(v)$ with respect to $\F^{-N_0}$.
From Lemma \ref{lemmasizeofRN}, whenever $\mathcal{F}^i(u)$ and $\F^i(v)$ belong to the same element of $\mathcal{R}$ for all $-N_0 \leq  i\leq N_0$, then $u$ and $v$ belong to the element of $\mathcal{R}_{N_0}$, and hence $d(u,v)<\ep$. We will use these facts repeatedly in the following lemma where we construct the required shadowing map $\psi$, and instead of the first return map $\F$, from now on we will work with its $N_0$-th power $\G:=\F^{N_0}.$

While we follow the classical proof of the shadowing lemma using the uniform hyperbolicity established in Lemma \ref{lem: hyperbolicity}, such hyperbolicity is only guaranteed contingent on the assumption that the vectors whose distance are being compared remain in the same elements of $\mathcal{R}$. In particular, at each step 
we have to ensure that relevant vectors in consideration belong to the suitable element of $\mathcal{R}$.

\subsection{Construction of the subshift} With the above choice of $\ep$ and $N_0$, the goal of this subsection is to build an injective coding map using the $N_0$-th power map $\G:=\F^{N_0}$. 
\begin{proposition}  \label{prop: psi injective}
Any $2N_0$-admissible sequence has a unique $\ep$-shadowing. Moreover, the shadowing map $\psi: \mathcal{A}_{2N_0}\to \C$ is injective. 
\end{proposition}

Let $a=(a_i)_{i\in \Z} \in \mathcal{A}_{2N_0}$ be $2N_0$-admissible. For each $i\in \Z$, let $u_i \in a_i\cap \Lambda $ such that $\mathcal{F}(u_i) \in a_{i+1}$ and $R^i$ be the element of $\mathcal{R}$ containing $a_i$. 
For simplicity of notations, we denote $$\u_i:=u_{iN_0} \text{ and } \RR^i:=R^{iN_0}$$ so that $\G$ maps $\u_i$ into $\RR^{i+1}$ for each $i \in \Z$. Moreover, we denote by $\D^i$ the element of $\mathcal{R}_{N_0}$ containing $\u_i$, whose diameter is bounded above by $\ep$ from Lemma \ref{lemmasizeofRN} and the choice of $N_0$.

The idea behind the construction of $\psi$ is quite simple which we briefly sketch in this paragraph. We first want to find $\w \in \RR^0 \cap \W^u(\u_0)$ such that $\G^n(\w)$ is well-defined belongs to $\G^n(\w) \in \RR^n$ for all $n\in \N$. Since $\G$ is uniformly hyperbolic on the domain in which it is well-defined, we can construct 
$\w$ as the limit of a (exponentially converging) Cauchy sequence $\{\G^{-n}(\w_n)\}_{n\in \N}$ on $\RR^0 \cap \W^u(\u_0)$ for some well-chosen $\w_n \in \RR^n$ whose image under $\G^{-n}$ is well-defined. Then we find $\v\in \RR^0 \cap \W^s(\u_0)$ satisfying the analogous properties with respect to $\G^{-1}$, and define $\psi(a)$ as the local product $[\v,\w]_\C$. It will then be easy to verify that $\psi(a)$ is the unique shadow of $a$, as claimed in the proposition.

For each $n\in \N_0$, we will first inductively construct $\w_{n} \in \RR^{n}$ satisfying the following properties: $$d(\w_n,\u_n)<\ep,$$ and for each $0 \leq j \leq n$, 
\begin{equation}\label{eq: 1}
\G^{-j}(\w_n) \in \RR^{n-j} \cap \W^u(\w_{n-j})
\end{equation}
and
\begin{equation}\label{eq: 2}
d(\G^{-j+1}(\w_{n-1}),\G^{-j}(\w_n))<\ep/2^{j}.
\end{equation}

Setting $\w_0 := \u_0$, the above listed properties (except for the last property which is irrelevant as we did not define $\w_{-1}$) are trivially satisfied for $n=0$. 

From the definition of $2N_0$-admissibility, $\G^i(\w_0)$ belongs to $\RR^{i}$ for all $j \in \{-2,-1,0,1,2\}$. In particular, $$\w_{1}:=[\G(\w_0),\u_1]_\C \in \RR^1$$ 
is well-defined.  For $j\in \{-2,-1,0,1\}$, $\G^j(\w_1)$ coincides with $[\G^{1+j}(\w_0),\G^{j}(\u_1)]_\C$ from Lemma \ref{lemmalpsforR}, and hence belongs to $\RR^{1+j}$. This implies that $\G(\w_0)$, $\w_1$, and $\u_1$ belong in the same element $\D^1$ of $\mathcal{R}_{N_0}$ and that $\G^{-1}(\w_1)\in \D^{0}$. Since the diameter of $\D^1$ is at most $\ep$, the choice of constant $N_0$ gives that $\displaystyle d(\w_0,\G^{-1}(\w_1)) < \frac{1}{2} d(\G(\w_0),\w_1) \leq \frac{\ep}{2}$. In particular, the above listed properties \eqref{eq: 1} and \eqref{eq: 2} for $\w_1$ hold for $n=1$:
$$d(\w_1,\u_1)<\ep,~d(\G(\w_0),\w_1)<\ep,\text{ and } d(\w_0,\G^{-1}(\w_1))<\ep/2.$$

Before moving onto the construction of $\w_2$, we establish another property of $\w_1$. 
From the above paragraph, we have $\G(\w_1) = [\G^2(\w_0),\G(\u_1)]_\C$ which belongs to $\mathsf{R}^2$. Since $\F$ is well-defined on each rectangle in $\mathcal{R}$, we know that $\F(\G(\w_1))$ is well-defined and belongs to $\C$. A priori, we do not know whether it belongs to one of the rectangles in $\mathcal{R}$ nor which rectangle it belongs to, if it belongs to one. 
This is because although we can write $\F(\G(\w_1))$ as $[\F(\G^2(\w_0)),\F(\G(\u_1))]_\C$ where $\F(\G(\u_1)) \in R^{2N_0+1}$, we do not know which rectangle $\F(\G^2(\w_0))$ belongs to, if it belongs to one; see Remark \ref{rem: F}. However, 
the following lemma shows that $\F(\G(\w_1))$ belongs to $R^{2N_0+1}$ as expected, and the same holds for $\F^j(\G(\w_1))$ for all $1 \leq j \leq N_0$.

\begin{lemma}\label{lem: w_1}
$\F^j(\G(\w_1)) \in R^{2N_0+j}$ for all $1 \leq j \leq N_0$. In particular, $\G^2(\w_1)\in \mathsf{R}^3$ is well-defined, and hence, $\G(\w_1) \in \D^2$.
\end{lemma}
\begin{proof}
We begin by noting that $\G(\w_1) \in \W^s(\G(\u_1))\cap \mathsf{R}^2$ and that $d(\G(\w_1),\G(\u_1)) \leq \ep/2$ from the choice of $N_0$.
By applying \eqref{eq: almost non-increase} with $n=1$ gives
$$d(\F(\G(\u_1)),\F(\G(\w_1))) \leq \b\k d(\G(\u_1),\G(\w_1)) \leq \b\k\ep/2 \leq \Delta.$$
In particular, from the defining property of $\Delta$ and the fact that $\F(\G(\u_1)) \in \Lambda$ we have $\F(\G(\w_1)) \in R^{2N_0+1}$. 
Therefore, $\F$ can be iterated for $\F(\G(\w_1))$ since $\F$ is well-defined on each rectangle in $\mathcal{R}$.
Now inductively applying the same argument using \eqref{eq: almost non-increase} with $n=2,\ldots, N_0$ proves the lemma.
\end{proof}

With such properties of $\w_1$ in mind, we define 
$$\w_2:=[\G(\w_1),\u_2]_\C \in \RR^2.$$
See Figure \ref{fig: 2} which also contains $\G^2(\w)$ where $\w$ is a vector (yet to be defined) described in the sketch of proof appearing below Proposition \ref{prop: psi injective}. 
As we will see in the following lemma, the statement $\G(\w_1) \in \D^2$ of the above lemma ensures that $d(\G(\w_1),\w_2)<\ep$ which corresponds to \eqref{eq: 2} of $\w_2$ for $j = 0$.
\begin{lemma} $\w_2$ satisfies \eqref{eq: 1} and \eqref{eq: 2}. Moreover, we have $\G(\w_2) \in \D^3$.
\end{lemma}
\begin{proof}
Using the fact that $a \in \mathcal{A}_{2N_0}$ is $2N_0$-admissible, we have $\G^{j}(\u_2) \in \RR^{2+j}$ for $-2 \leq j \leq 2$. From the definition of $\w_1$ and Lemma \ref{lem: w_1}, we have $\G^{1+j}(\w_1) \in \RR^{2+j}$ for $-2 \leq j \leq 1$. In particular, for the same range of $j$,
$\G^j(\w_2)$ is defined as $[\G^{1+j}(\w_1),\G^{j}(\u_2)]_\C$, and hence belongs to $\RR^{2+j}$. This implies that $\G(\w_1)$, $\w_2$, and $\u_2$ belong in the same element $\D^2$ of $\mathcal{R}_{N_0}$ and that $\w_1$ and $\G^{-1}(\w_2)$ belong in $\D^1$. Then the uniform hyperbolicity of $\G$ coming from the choice of constant $N_0$ gives \eqref{eq: 2} for $\w_2$: 
 $$d(\w_2,\u_2)<\ep, \text{ and } d(\G^{-j+1}(\w_1),\G^{-j}(\w_2))<\ep/2^j \text{ for } j \in \{0,1,2\}.$$
The remaining statement that $\G(\w_2)$ belongs to $\D^3$ follows just as in Lemma \ref{lem: w_1}.
\end{proof}

We describe another iteration prior to generalizing this process. Let
$$\w_3:=[\G(\w_2),\u_3]_\C \in \RR^3.$$
\begin{lemma} $\w_3$ satisfies \eqref{eq: 1} and \eqref{eq: 2}.
\end{lemma}
\begin{proof}
Proceeding as in the above lemma, $\G^j(\w_3)$ belongs to $\RR^{3+j}$ for all $j\in \{-2,-1,0,1\}$.
Moreover, $\G(\w_2)$, $\w_3$, and $\u_3$ belong in the same element $\D^3$ of $\mathcal{R}_{N_0}$ and that $\w_2$ and $\G^{-1}(\w_3)=[\w_2,\G^{-1}(\u_3)]_\C$ belong in $\D^2$. Lemma \ref{lemmasizeofRN} and the choice of constant $N_0$ then give
$$
 d(\w_3,\u_3)<\ep, \text{ and } d(\G^{-j+1}(\w_2),\G^{-j}(\w_3))<\ep/2^j \text{ for } j \in \{0,1,2\},
$$
where $\G^{-2}(\w_3)=[\G^{-1}(\w_2),\G^{-2}(\u_3)]_\C$ from Lemma \ref{lemmalpsforR}.

A priori, this is all that can be deduced from the construction of $\w_3$; that is, we do not quite have the well-definedness of $\G^{-3}(\w_3)$ nor the properties corresponding to \eqref{eq: 1} and \eqref{eq: 2} for $j=3$.
This is because $a=(a_i)_{i\in \Z} \in \mathcal{A}_{2N_0}$ is only $2N_0$-admissible, so $\F^{-j}(\G^{-2}(\u_3))$ does not necessarily belong to $R^{N_0-j}$ for $1 \leq j \leq N_0$, and hence, $\F^{-j}(\G^{-2}(\w_3))$ cannot be defined as $[\F^{-j}(\G^{-1}(\w_2)),\F^{-j}(\G^{-2}(\u_3))]_\C$ via Lemma \ref{lemmalpsforR}. Instead, we can directly show using the definition of $\Delta$ that 
\begin{equation}\label{eq: correct rectangle}
\F^{-j}(\G^{-2}(\w_3)) \in R^{N_0-j} \text{ for all } 1\leq j \leq N_0.
\end{equation}

Indeed, from the fact $d(\G^{-1}(\w_2),\G^{-2}(\w_3)) <\ep/4$ established above, there exists a small $t_0\in \R$ such that $g_{t_0}\G^{-2}(\w_3)\in W^u_{\k\ep/4}(\G^{-1}(\w_2))$ by \eqref{eq: d and d^s}.
Since the $d^u$-distance is non-increasing with respect to the backward geodesic flow, we have
\begin{equation}\label{eq: projection}
g_{t_0-\tau_0}\G^{-2}(\w_3)\in W^u_{\k\ep/4}(\F^{-1}(\G^{-1}(\w_2)))
\end{equation}
where $\tau_0 = \tau(\F^{-1}(\G^{-1}(\w_2)))$.

Since $d(\F^{-1}(\G^{-1}(\w_2)),\F^{-1}(\w_1))<\ep/2$, $d(\F^{-1}(\w_1),\F^{-1}(\G(\w_0)))<\ep$, and $\F^{-1}(\G(\w_0))$ belongs to $\Lambda \cap R^{N_0-1}$, we have from the defining property of $\Delta$ that
$$
\W^u_{\b\k\ep/4}(\F^{-1}(\G^{-1}(\w_2))) \subset R^{N_0-1}
$$
because $\ep+\ep/2+\b\kappa\ep/4<2\ep+\b\kappa\ep/4<\Delta$.
All vectors in $\W^u_{\b\k\ep/4}(\F^{-1}(\G^{-1}(\w_2)))$ project to $W^u(\F^{-1}(\G^{-1}(\w_2)))$ along the flow direction, and by \eqref{eq: d and d^s} the image of such a projection contains  $W^u_{\k\ep/4}(\F^{-1}(\G^{-1}(\w_2)))$.
From \eqref{eq: projection} there exists $t_1\in \R$ such that $g_{t_0+t_1-\tau_0}\G^{-2}(\w_3)$ belongs to $R^{N_0-1}$, and hence, must be equal to $\F^{-1}(\G^{-2}(\w_3))$. This establishes \eqref{eq: correct rectangle} for $j=1$, and by repeating the argument, we can show \eqref{eq: correct rectangle} for the other $j$'s inductively. In particular, it shows that $\G^{-3}(\w_3)$ belongs to $\mathsf{R}^0 \cap \W^u(\w_0)$ and $d(\G^{-2}(\w_2),\G^{-3}(\w_3))<\ep/8$ from the definition of $N_0$. These correspond to \eqref{eq: 1} and \eqref{eq: 2} of $\w_3$ with $j=3$.
\end{proof}

Suppose now that $\w_0,\w_1,\ldots,\w_{n-1}$ are constructed with the listed properties \eqref{eq: 1} and \eqref{eq: 2} for some $n\in \N$.
For the general inductive step, we set
$$\w_{n}:=[\G(\w_{n-1}),\u_{n}]_\C$$ and verify the listed properties above in the following lemma:
\begin{lemma}\label{lem: inductive}
For all $0 \leq j \leq n$, $\G^{-j}(\w_n)$ is well-defined and satisfies \eqref{eq: 1} and \eqref{eq: 2}.
\end{lemma}
\begin{proof} The proof resembles that of $\w_3$ above. Proceeding as above, $\G(\w_{n-1})$, $\w_{n}$, and $\u_{n}$ belong in the same element $\D^{n}$ and that $\w_{n-1}$ and $\G^{-1}(\w_{n})$ belong in $\D^{n-1}$. Lemma \ref{lemmasizeofRN} and the choice of $N_0$ then give
$$
 d(\w_{n},\u_{n})<\ep, \text{ and } d(\G^{-j+1}(\w_{n-1}),\G^{-j}(\w_{n}))<\ep/2^j \text{ for } j \in \{0,1,2\}.
$$
This proves \eqref{eq: 2} for $j \in \{0,1,2\}$. 

For the inductive step, suppose $\G^{-k}(\w_n) \in \mathsf{R}^{n-k} \cap \W^u(\w_{n-k})$ and $d(\G^{-k+1}(\w_{n-1}),\G^{-k}(\w_n))<\ep/2^k$ for some $k\geq 2$. As in the proof of \eqref{eq: correct rectangle}, but instead using $\F^{N_0-j}(\u_{n-k-1})$ as the vector lying in $\Lambda$ and pivoting at $\F^{N_0-j}(\w_{n-k-1})$, we can show
$$\F^{-j}(\G^{-k}(\w_n))\in R^{(n-k)N_0-j} \text{ for all } 1 \leq j \leq N_0$$
by observing that $\displaystyle \ep+ \sum\limits_{i=0}^{k-1}\ep/2^i+\b\k\ep/2^k<2\ep+\b\kappa\ep/2^k<\Delta$. In particular, $\G^{-k-1}(\w_n)$ is well-defined and belongs to $\mathsf{R}^{n-k-1} \cap\W^u(\w_{n-k-1})$, and it follows from the definition of $N_0$ that $d(\G^{-k}(\w_{n-1}),\G^{-k-1}(\w_n))<\ep/2^{k+1}$, completing the inductive process.
\end{proof}

As in the proof of Lemma \ref{lem: w_1} and \eqref{eq: correct rectangle}, we have $\F^j(\G^{-n}(\w_n))\in R^j$ for all $-N_0 \leq j \leq 0$ and $nN_0 \leq j \leq (n+1)N_0$. It then follows from Lemma \ref{lemmasizeofRN} and \eqref{eq: 1} that
$$d(\F^j(\G^{-n}(\w_n)),u_j) <\ep \text{ for all }0 \leq j \leq nN_0.$$

From the obtained sequence $\{\w_n\}_{n\in \N}$ satisfying \eqref{eq: 1} and \eqref{eq: 2}, we define 
$$\w:=\lim\limits_{n \to \infty} \G^{-n}(\w_n),$$
which belongs to $ \RR^0 \cap \W^u(\u_0)$.
Here the convergence of the limit is guaranteed because $\G^{-n}(\w_n)$ forms a Cauchy sequence due to \eqref{eq: 2}. Moreover, both $d(\u_0,\w)$ and $d(\G^{n}(\w),\w_n)$ are bounded above by $ \sum\limits_{i=1}^\infty 2^{-i}\ep= \ep$. From Lemma \ref{lemmasizeofRN} we then have $d(\F^j(\w),u_j)<\ep$ for all $j\in \N$.

We repeat the construction using the negative indices of $a = (a_i)_{i\in \Z}\in \mathcal{A}_{2N_0}$ and this gives us a sequence $\{\v_n\}_{n \leq 0}$ satisfying the analogous properties as $\{\w_n\}_{n \geq 0}$ with respect to $\W^s$ instead. In particular, we may define
$$\v:=\lim\limits_{n\to \infty}\G^n(\v_n),$$
which belongs to $\RR^0 \cap \W^s(\u_0)$.
Finally, the desired shadowing map $\psi$ can be defined as
$$\psi(a) := [\v,\w]_\C.$$
See Figure \ref{fig: 3}.

\begin{figure}[h]\label{fig: 3}
  \centering
    \includegraphics[width=0.75\textwidth]{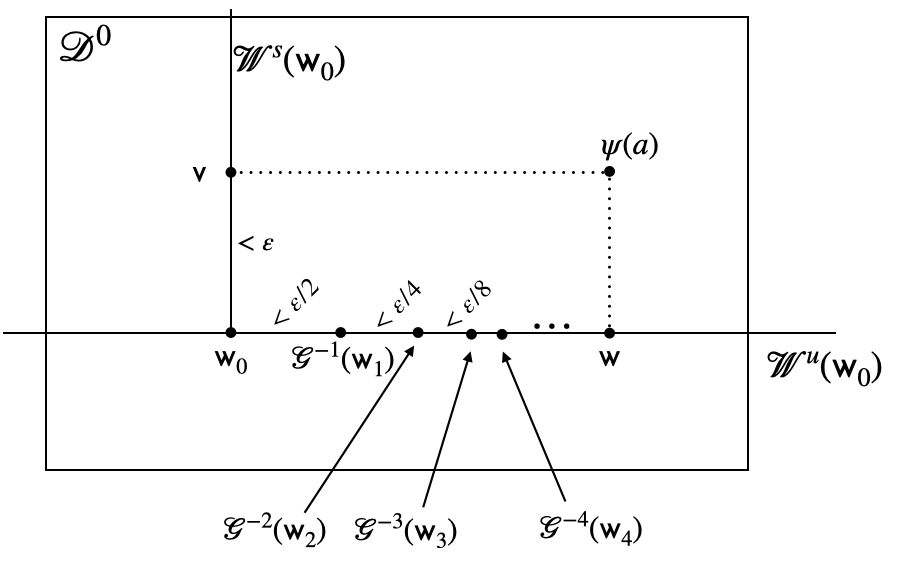}
    \caption{}
\end{figure}

Since $d(\u_0,\w)<\ep$ and $d(\u_0,\v)<\ep$, we have $d(\w,\v)<2\ep$ and hence $d(\w,\psi(a)) \leq 2\b\k\ep$ from \eqref{eq: d relation}.
From \eqref{eq: almost non-increase}, we have $d(\F^j(\w),\F^j(\psi(a)))\leq \b\k d(\w,\psi(a)) \leq 2\b^2\k^2\ep$ for all $j\in \N$,
and hence,
$$d(u_j,\F^j(\psi(a))) \leq d(u_j,\F^j(\w))+d(\F^j(\w),\F^j(\psi(a)) \leq \ep+2\b^2\k^2\ep.$$
Since $u_j\in \Lambda$ and $\ep+2\b^2\k^2\ep <\Delta$, we have $\F^j(\psi(a)) \in R^j$ for all $j\in \N$. 
The analogous inequality for the negative indices $n\leq 0$ can similarly be verified using $\v_n$ and $\G^n(\v)$, and this shows that $\F^j(\psi(a))\in R^j$ for all $j\in \Z$. Then it follows from Lemma \ref{lemmasizeofRN} that $\psi(a)$ indeed $\ep$-shadows $a$ with respect to $\F$, as required.

The well-definedness and injectiveness of $\psi$ can also easily be verified. For well-definedness, if we had used a different sequence $\{u_i'\}_{i\in \Z}$ for $a = (a_i)_{i\in \Z}$ and obtained $\psi(a)' \in \RR^0$, then both $\G^n(\psi(a))$ and $\G^n(\psi(a)')$ belong to $\RR^n$ for all $n\in \Z$, and this implies that $\psi(a)=\psi(a)'$ from Lemma \ref{lemmasizeofRN}. The injectivity of the map $\psi$ follows because the orbit of $\psi(a)$ can conversely be used to determine $a$. 
This completes the proof of Proposition \ref{prop: psi injective}.

\subsection{Completing the proof of Proposition \ref{prop: key}}
Using $\psi$, we construct $\wt\Lambda$ as follows:
$$ \wt\Lambda:=\bigcup_{t\in \mathbb{R}}g_t(\psi(\mathcal{A}_{2N_0})).$$ 
We claim that $\wt\Lambda$ is the desired set satisfying the statements of Proposition \ref{prop: key}. Indeed, every vector in $\Lambda$ is contained in $\mathcal{A}_{2N_0}$ and mapped to itself by $\psi$, so $\Lambda\subset \wt\Lambda$. From its construction, $\wt\Lambda$ is $G$-invariant and compact as it is the image of a compact set under the continuous map $\psi$. Moreover, $\wt\Lambda$ is uniformly hyperbolic as it contained in $ U$, and it is locally maximal because it inherits the local product structure of $\mathcal{A}_{2N_0}$. This completes the proof of Proposition \ref{prop: key}.

\bibliographystyle{amsalpha}
\bibliography{bg_generalize}
\end{document}